\documentclass[12pt,letterpaper]{amsart}
\usepackage{cite}
\usepackage{graphicx}
\usepackage[T1]{fontenc}
\usepackage{amssymb,amscd,amsthm,amsxtra}
\usepackage{latexsym}
\usepackage{epsfig}
\usepackage{amsmath}
\usepackage{bm}
\usepackage{amssymb}
\usepackage{amsthm}
\usepackage[usenames,dvipsnames,svgnames,table]{xcolor}
\usepackage[linktocpage=true,colorlinks=true,linkcolor=Blue,citecolor=BrickRed,urlcolor=RoyalBlue]{hyperref}
\usepackage{comment}
\usepackage[colorinlistoftodos,prependcaption,textsize=tiny]{todonotes}
\usepackage{caption}
\usepackage{amscd}
\usepackage[shortlabels]{enumitem}
\usepackage[mathscr]{euscript} 
\usepackage{mathtools, verbatim}
\usepackage[utf8]{inputenc}
\usepackage{indentfirst}
\usepackage{enumerate}
\usepackage[colorlinks=true]{hyperref}
\hypersetup{urlcolor=blue, citecolor=red}
\usepackage{graphicx}
\usepackage[usenames,dvipsnames]{pstricks}
\usepackage{epsfig}
\usepackage{pst-grad} 
\usepackage{pst-plot} 
\usepackage{ifthen}
\usepackage{color, xcolor}
\usepackage{soul}

\usepackage{mathrsfs} 

\usepackage{pifont}

\newlist{thmlist}{enumerate}{1}
\setlist[thmlist]{label=(\alph{thmlisti}), ref=\thethm.(\alph{thmlisti}),noitemsep}

\usepackage{thmtools}

\declaretheorem[
    name=Theorem,
    Refname={Theorem,Theorems},
    numberwithin=section]{thm}
\declaretheorem[
    name=Lemma,
    Refname={Lemma,Lemmas},
    sibling=thm]{lem}

    \declaretheorem[
    name=Proposition,
    Refname={Proposition,Propositions},
    sibling=thm]{prop}

\usepackage{nameref,hyperref}
\usepackage[capitalize]{cleveref}

\Crefname{thm}{Theorem}{Theorems}
\Crefname{lem}{Lemma}{Lemmas}

\makeatletter
\newcommand{\mylabel}[2]{#2\def\@currentlabel{#2}\label{#1}}
\makeatother

\numberwithin{equation}{section} 

\DeclareMathOperator{\Tr}{Tr}

\newtheorem{theorem}{Theorem}

\newtheorem{definition}{Definition}

\newtheorem{corollary}{Corollary}

\theoremstyle{remark}
\newtheorem{remark}{Remark}[section]

\newtheorem{Assumption}{A}

\title[Optimal regularity for elliptic equations]{Optimal regularity for degenerate elliptic equations with Hamiltonian terms}
\author[P. D. S. Andrade]{P\^edra D. S. Andrade}
\address{Department of Mathematics, Paris Lodron Universit\"at Salzburg, Hellbrunnerstrasse 34, A-5020 Salzburg, AT}{}
\email{pedra.andrade@plus.ac.at}

\author[T. M. Nascimento]{Thialita M. Nascimento}
\address{Department of Mathematics, Iowa State University, 396 Carver Hall, 50011, Ames, IA, USA}{}
\email{thnasc@iastate.edu}


\begin{document}

\date{} 

\begin{abstract} 
We establish optimal, quantitative Hölder estimates for the gradient of solutions to a class of degenerate elliptic equations with Hamiltonian terms. The presence of such lower-order terms introduces additional challenges, particularly in regimes where the gradient is either very small or very large. Our approach adapts perturbative techniques to capture the interplay between the degeneracy rate and the Hamiltonian’s growth. Our results naturally extend the regularity theory developed by Araujo–Ricarte–Teixeira (Calc. Var. 53:605–625, 2015) and Birindelli–Demengel (Nonlinear Differ. Equ. Appl. 23:41, 2016) to a  more general setting.

\tableofcontents
\end{abstract}

\maketitle

\section{Introduction}

In this paper, we study the interior regularity of viscosity solutions for a class of second-order degenerate fully nonlinear elliptic equations of the form 
\begin{equation}\label{main eq}
\Phi(x, \nabla u) F(x, D^2 u) + h(x) |\nabla u|^m = f(x) \quad \text{in} \quad B_1,
\end{equation}  
where \(F\) is a fully nonlinear, uniformly elliptic operator, $f$ and $h$ are bounded functions and \(\Phi : B_1 \times \mathbb{R}^n \to \mathbb{R}\)  is the function that captures the degeneracy law,  $\Phi( x, \nabla u ) \sim |\nabla u|^{\gamma}$  with \(\gamma > 0\), and $ 0 < m \le 1 + \gamma$. The considered equation includes, as a very special case, the semilinear linear equation
$$
    \Delta u + h(x) |\nabla u |^m = f(x)
$$
with $0 < m \le 1$, and for this reason,  the Hamiltonian term in equation \eqref{main eq} is said to be \emph{sublinear} when \(0 < m < 1 + \gamma\) and \emph{linear} when \(m = 1 + \gamma\) as in \cite{BD2016}. 

Equations of this type encompass and generalize several models studied in the literature, particularly within the context of degenerate elliptic and Hamilton-Jacobi equations; see, for example, \cite{JPU22, LP16} and references therein. Regularity results for degenerate diffusive Hamilton-Jacobi equations are also essential in a variety of applications, including global continuation, ergodic, and homogenization problems; see, e.g., \cite{AttBarles15, Barles10}.

The modern theory of regularity for fully nonlinear uniformly elliptic equations is well developed. Since the foundational works of Krylov and Safonov~\cite{Krylov-Safonov1979, Krylov-Safonov1980, Safonov}, it is known that viscosity solutions to the homogeneous constant coefficient equation  
\begin{equation}\label{the constant homog model}
F(D^2 z) = 0
\end{equation}  
are locally \(C^{1,\alpha_0}\) for a universal exponent, $\alpha_0 \in (0,1)$, depending only on the dimension and the ellipticity constants of $F$. If no further structure is imposed on the operator $F$, this regularity is, in fact, optimal; see e.g., \cite{NV2007, NV2008, NV2011}. When ellipticity degenerates as a power of the gradient, however, the standard Krylov-Safonov theory fails and new methods are needed.

Important progress in this direction was made by Imbert and Silvestre~\cite{ImbertSilvestre2013}, who developed a regularity theory for degenerate elliptic equations driven by the gradient’s magnitude. Subsequent works, such as~\cite{ART, BD2016, BDL2019, Baasandorj-Byun-Oh2023}, extended these results to more general settings, considering either variable coefficients or lower-order terms. However, the simultaneous presence of both features remains less explored.

A notable contribution was obtained in \cite{ART}, where the authors established sharp gradient regularity estimates for degenerate elliptic equations. Specifically, they showed that solutions to the varying coefficient, degenerate equation
$$
    |\nabla u |^{\gamma} F(x, D^2 u) = f(x)
$$
with $f$ being a bounded fucnction, are locally $C_{loc}^{1, \min\{\alpha_0, \frac{1}{1+\gamma}\}}$, under continuity assumption of the coefficients. 

Regularity results for degenerate equations involving Hamiltonian terms were first studied by Birindelli and Demengel in \cite{BD2014} and later extended in \cite{BD2016}. In these works, the authors treated the constant coefficient case and showed that solutions belong to \( C_{loc}^{1,\alpha} \) for some unspecified exponent \(\alpha \in (0,1)\).

The main goal of this work is to establish sharp interior regularity estimates for degenerate elliptic equations with nontrivial Hamiltonian terms in heterogeneous media. Our method builds on the flatness improvement scheme in \cite{ART}, but the interplay between the degeneracy and the Hamiltonian’s growth creates a delicate balance and requires a new analysis tailored to our setting.

A crucial step is to develop robust compactness and approximation tools that remain valid under gradient perturbations. To this end, it is necessary to understand how the interplay between the degeneracy law and the Hamiltonian growth affects the structure of the equation. In regions where the gradient is small, the equation is essentially degenerate, and existing tools for such models can be applied, as in \cite{ART} and \cite{ImbertSilvestre2013}. However, away from the critical set \(\mathcal{C}(u)\), the equation can, loosely speaking, be rewritten in Hamilton-Jacobi form:
\begin{equation}\label{heuristics}
    F(x, D^2 u) + h(x) |\nabla u|^{m - \gamma} = f(x) |\nabla u|^{-\gamma},
\end{equation}
which highlights the challenge of controlling the growth of the Hamiltonian term under gradient perturbations.

To overcome this issue, we adopt pivotal scaling and approximation strategies inspired by~\cite{APPT, SN}. The analysis is centered on the interplay between the exponent regimes of the degeneracy law and the Hamiltonian term. When \( m \leq \gamma \), the heuristic form in \eqref{heuristics} shows that the equation is essentially elliptic for large gradients, allowing us to treat it as in \cite{ART}. In contrast, when \( \gamma < m \leq 1 + \gamma \), additional conditions are needed to ensure that the Hamiltonian grows slowly enough so that, under a suitable smallness regime, solutions to the perturbed equation can be approximated by solutions of the homogeneous constant-coefficient equation. Our analysis adapts the techniques developed in~\cite{ART, BD2016} to the setting of variable coefficients. The resulting estimates will depend quantitatively on the modulus of continuity of the map \( x \mapsto F(x, X) \).

The paper is organized as follows. In Section~\ref{prelim}, we state the assumptions, fix the notation, and present the auxiliary results. Section~\ref{main_results} is devoted to presenting the main results. Section~\ref{compt resul}, we establish compactness results for the perturbed equation. Section~\ref{sct approx} contains the key approximation lemma. Finally, in section~\ref{grad reg}  we complete the proof of the sharp interior gradient estimates.

\section{Preliminaries}\label{prelim}
In this section, we fix the notation and collect definitions and auxiliary results from the classical theory that will be used throughout the paper. Here $\mathbb{R}^n$ stands for the $n$-dimensional Euclidean space and ${\rm Sym}(n)$ denotes the space of real $ n\times n$ symmetric matrices. Also, we denote $B_r(x_0) \subset \mathbb{R}^n$ to be the open ball with radius $r>0$ and centered at $x_0$. For simplicity, when $x_0 = 0$, we denote $B_r(0) = B_r$.  In what follows, we define the uniform ellipticity condition of the operator $F$.

\begin{definition}
    An operator $F: B_1 \times {\rm Sym}(n) \to \mathbb{R}$ is said to be uniformly elliptic if there are two positive constants $ \lambda \le \Lambda$ such that for any $x \in B_1$ and $M, N \in {\rm Sym}(n)$, there holds
\begin{equation}\label{unif ellipticity}  
    \mathcal{M}_{\lambda,\Lambda}^-(M-N) \leq F(x,M) - F(x,N) \leq \mathcal{M}_{\lambda,\Lambda}^+(M-N), 
\end{equation}
where $\mathcal{M}^+_{\lambda,\Lambda}$ and $\mathcal{M}^-_{\lambda,\Lambda}$ are the {\it Pucci's extremal operators}. In addition, we call $\lambda$ and $\Lambda$ ellipticity constants.
\end{definition}

In the remainder of this paper, we refer to any operator $F$ satisfying \eqref{unif ellipticity} as a uniformly $(\lambda, \Lambda)$-elliptic operator. 

For normalization purposes,  we assume, with no loss of generality, throughout the text that that $F(x, 0) \equiv 0$ for all $x \in B_1$.

We now define the {\it Pucci's extremal operators} introduced by Pucci in \cite{Pucci}.
\begin{definition}
   Let $\lambda \le \Lambda$ two positive constants. For $M \in {\rm Sym}(n)$, we define 
   $$
    \mathcal{M}_{\lambda,\Lambda}^{-} (M)  =  \inf \{\Tr(AM) \colon A \in {\rm Sym}(n) \, \, \text{and} \, \,  \lambda I\leq A \leq \Lambda I  \},
$$
and 
   $$
    \mathcal{M}_{\lambda,\Lambda}^{+} (M)  =  \sup \{\Tr(AM) \colon A \in {\rm Sym}(n) \, \, \text{and} \, \,  \lambda I\leq A \leq \Lambda I  \},
$$
where $\Tr(M)$ is the trace of the matrix $M$ and $I$ denotes the $n \times n$ identity matrix.
\end{definition}

 For an operator $G: B_1 \times \mathbb{R}^n \times {\rm Sym}(n) \to \mathbb{R}$, solutions to 
 $$
    G(x, \nabla u, D^2 u) = 0\, \quad  \text{in} \, \quad B_1
$$ are understood in the sense of viscosity, as defined in  \cite{Caff-Cabre} and \cite{CIL92}. More precisely, we have

\begin{definition}[Viscosity solution]
Let $G \colon B_1 \times \mathbb{R}^n \times \mbox{\rm Sym}(n) \to \mathbb{R}$ be a degenerate elliptic operator. We say that $u$ is a viscosity subsolution to
$$
    G(x,Du,D^2u)=0 \, \quad  \text{in} \, \quad B_1,
$$
if for every $x_0 \in B_1$ and $\varphi \in C^2\left(B_1\right)$ such that $ u -\varphi$ attains a local maximum at $x_0$, we have 
$$
    G(x_0, D\varphi(x_0), D^2\varphi(x_0)) \le  0.
$$
Conversely, we say that $u$ is a viscosity supersolution to
$$
   G(x,Du,D^2u)=0 \, \quad  \text{in} \, \quad B_1,
$$
if for every $x_0 \in B_1$ and  $\varphi \in C^2\left(B_1 \right)$ such that $u -\varphi$ attains a local minimum at $x_0$, we have 
$$
   G(x_0, D\varphi(x_0), D^2\varphi(x_0)) \ge 0.
$$
A function $u$ is said to be a viscosity solution if it is both a subsolution and a supersolution. 
\end{definition}

We say that a viscosity solution $u$ is {\it normalized} if $\| u \|_{L^{\infty}(B_1)} \le 1$.   For a comprehensive account of the theory of viscosity solutions, we refer the reader to \cite{Caff} and \cite{Caff-Cabre}.

While maximum principle and existence of viscosity solutions were obtained in \cite{BD04, BD06, BD10}, it is worth mentioning that, in general, uniqueness of the Dirichlet problem does not hold when the Hamiltonian is not Lipschitz continuous. For more details on this matter, see \cite{BD2016}.

Next, we detail our main assumptions. We begin by stating the growth condition imposed on $\Phi$.
\begin{Assumption}[Degeneracy law of $\Phi$]
Let $\bar{\lambda}< \bar{\Lambda}$ be positive constants and $\gamma > 0$. We assume that $\Phi$ satisfies
    \begin{equation}\label{degeneracy law}
    {\bar{\lambda}} |\xi|^{\gamma} \le \Phi(x, \xi) \le {\bar{\Lambda}} |\xi |^{\gamma},  
\end{equation}
for all $x \in B_1 $, and $ \xi \in \mathbb{R}^n .$
\end{Assumption}

 Since the model in \eqref{main eq} depends on $x$, it is necessary to account for the oscillation of the operator $F$.

\begin{definition}\label{oscillation-F}
Let $M \in {\rm Sym}(n)$. The {\it oscillation} of $F$ in the variable $x$ is given by
\[
{\rm osc}_F (x, y) := \sup\limits_{\|M\| \le 1} \frac{\left| F(x, M) - F(y, M) \right| }{\|M\|+ 1} 
\]
 for all $x, y \in B_1$. In addition, we denote ${ \rm osc}_F (x) = {\rm osc}_F (x, 0)$.
\end{definition}

For equations with varying coefficients, the continuity assumption on the coefficients is an essential component in establishing the regularity of the gradient of the solutions. See, for example, \cite{Caff}. Henceforth, we shall assume the following uniform continuity condition on the coefficients of $F$.

\begin{Assumption}[H\"older continuity of the coefficients]\label{oscillation condition}
We assume that for some $\tau \in (0,1)$ there holds
  \begin{equation}\label{coeff cont}
    \text{osc}_F (x, y) \le C |x - y|^{\tau} \, \, \text{for all}\, \, x, y \in B_1, 
\end{equation} 
for some  constant $C >0$. 
\end{Assumption}
For convenience, we denote 
$$ 
    C_F = \inf \{ C > 0 : \text{osc}_F (x,y) \le C |x -y|^{\tau}\, x, y \in B_1 \}.
$$
 As a consequence of assumption A~\ref{oscillation condition}, we obtain the uniform continuity for the coefficients of $F$. That is,
\begin{equation}\label{uniform-continuityF}
 \left| F(x, M) - F(y, M) \right| \le C_F|x - y|^{\tau} (\|M\|+ 1),
\end{equation}
for every $M \in {\rm Sym}(n)$, and for all $x, y \in B_1$.

 We recall that, see e.g., \cite{Caff, Caff-Cabre,  Teix2014}, that under such continuity condition on the coefficients, viscosity solutions to
$$
    F(x,D^2 u)= f(x) \quad \text{in} \quad  B_1
$$
with $f$ being a bounded function in $B_1$,  are locally of class $C^{1, \beta}$, for any $0 < \beta < \alpha_{0}$, where $\alpha_{0}$ is the optimal H\"older exponent coming from the  classical Krylov-Safonov regularity result, for solutions to constant coefficient, homogeneous equation $F(D^2 h) = 0$, as mentioned in the introduction. See, for instance, \cite[Chapter 8]{Caff-Cabre}.
 

The degenerate case with varying coefficients was studied in \cite{ART}, where the authors established optimal gradient continuity estimates for viscosity solutions to fully nonlinear, degenerate elliptic equations. Their work provides sharp regularity results, showing that solutions are \( C_{loc}^{1, \min(\alpha_0, \frac{1}{1+\gamma})} \), where \( \alpha_0 \) is the Hölder regularity exponent for uniformly elliptic equations, and \( \gamma \) characterizes the degree of degeneracy. The continuity of the coefficients is likewise essential for ensuring gradient regularity and for controlling the qualitative behavior of solutions in degenerate settings.

\subsection{ Auxiliary results}

In what follows, we state the well-known Crandall-Ishii-Lions lemma. For a proof of this lemma, we refer the reader to {\cite[Proposition II.3]{IL}}, see also \cite[Theorem 3.2]{CIL92}.

\begin{lem}[Crandall-Ishii-Lions Lemma]\label{Crandall-Ishii-Lions}
    Let $G \colon B_1 \times \mathbb{R}^n \times \mbox{\rm Sym}(n) \to \mathbb{R}$ be a degenerate elliptic operator. Let $\Omega \subset B_1$, $u \in C(B_1)$, and $\psi \in C^2$ in a neighborhood of $\Omega \times \Omega$. Set 
    $$
        \omega(x,y) = u (x) - u(y), \quad \text{for} \,\, (x,y) \in \Omega \times \Omega.
    $$
    If $\omega - \psi$ attains its maximum at $(\bar{x}, \bar{y}) \in \Omega \times \Omega$, then for each $\varepsilon > 0$, there exists $X, Y \in {\rm Sym} (n)$ such that 
    $$
        G(\bar{x}, D_x \psi (\bar{x}, \bar{y}), X) \le 0 \le   G(\bar{y}, D_y \psi (\bar{x}, \bar{y}), Y).
    $$
    Moreover, $X,Y$ satisfy
    $$
        - \left( \frac{1}{\varepsilon} + \| A\| \right) I \le \begin{pmatrix}
            X & 0 \\
            0 & - Y
        \end{pmatrix} \le A + \varepsilon A^2
    $$
    where $A = D^2 \psi (\bar{x}, \bar{y})$. 
\end{lem}

\section{Main result} \label{main_results}

In this section, we state and comment our main result. In order to grasp the expected optimal regularity of the solutions to the equation \eqref{main eq}, we present a heuristic analysis on a simpler case. Consider the equation
\begin{equation}
    |\nabla u|^{\gamma} \Delta u + h(x) |\nabla u|^m = f(x). 
 \end{equation}
 We will infer optimal regularity by a scaling argument. Define
 \[
v(x) = \frac{u(rx)}{r^{\theta}}.
\]
Immediate computations shows that $v$ solves 
$$
    |\nabla v|^{\gamma} \Delta v + \tilde{h} (x) |\nabla v|^m  = \tilde{f} (x)
$$
where 
$$
    \tilde{h}(x) = r^{2 + \gamma - m - \theta (1 + \gamma - m)} h(r x ) \quad \text{and} \quad \tilde{f}(x) = r^{2 + \gamma - \theta (1 + \gamma)} f(r x). 
$$
Therefore, if we choose
$$
   0 < \theta \le \min\left\{ \frac{2 + \gamma}{1 + \gamma} , \frac{2 + \gamma - m}{1 + \gamma - m} \right\}, 
$$
then 
\[
\| \tilde{ h}\|_{L^{\infty} (B_1)} \le \| h\|_{L^{\infty} (B_1)}  \, \, \,  \text{and}\, \,\, \| \tilde{ f}\|_{L^{\infty} (B_1)} \le \| f \|_{L^{\infty} (B_1)}.
\]
Hence $v$ satisfies an equation with the same structure as $u$.
Given that \( m > 0 \), then it follows immediately that \( \frac{2 + \gamma}{1+\gamma} \) is the lesser quantity. Moreover, we note that 
$$
    \frac{2 + \gamma}{1 + \gamma} = 1 + \frac{1}{1 + \gamma}.
$$
Hence,  the optimal regularity is expected to be $C^{1, \frac{1}{1+\gamma}}$. 

Motivating for the above discussion, we are able to state our main result.

 \begin{theorem}\label{main thm}
    Let $u \in C(B_1)$ be a normalized viscosity solution to 
    $$
    \Phi( x, \nabla u )F(x, D^2 u) + h(x) |\nabla u|^{m}  = f(x) \quad \text{in} \quad B_1.
$$
   Assume that $f \in L^{\infty}(B_1)$, $h \in C(B_1)\cap L^{\infty} (B_1)$, and that $\Phi$ satisfies \eqref{degeneracy law}.  Assume further that  \eqref{unif ellipticity} and \eqref{coeff cont} are in force.  Fixed an exponent 
   $$
    \beta \in (0, \alpha_0 ) \cap \left( 0 , \frac{1}{1 + \gamma} \right], 
   $$
   there exist positive constants $C_{\beta}$,  $\bar{C}_{\beta}$, depend only on $ \beta, n, \lambda, \Lambda, \gamma, m, \omega$, such that  $u \in C^{1,\beta}(B_{1/2})$, with the following estimates: for $m < 1 + \gamma$,  
    \begin{equation}\label{main estim}
        \| u\|_{C^{1,\beta}(B_{1/2})} \le C_{\beta} \left(  \| u\|_{L^{\infty}(B_1) }  + \| h\|_{L^{\infty}(B_1)}^{\frac{1}{1+ \gamma - m}} + \|f \|_{L^{\infty} (B_1)}^{\frac{1}{1+\gamma}} \right) ,
    \end{equation}
     and for  $m = 1+\gamma$, 
    \begin{equation}\label{linear main estim}
         \| u\|_{C^{1,\beta}(B_{1/2})} \le \bar{C}_{\beta}  \| u\|_{L^{\infty}(B_1)},
    \end{equation}
    where $\bar{C}_{\beta} $ depends additionally on $\|f\|_{L^{\infty} (B_1)}$ and $\|h\|_{L^{\infty} (B_1)}$. 
\end{theorem}
 
Theorem \ref{main_results} extends  and improves the results  in \cite{ART}, \cite{BD2014} and \cite{BD2016}. It also revels the regularizing effect of the Hamiltonian term $h(x)|\nabla u|^m$, $m > 0$. In fact, the addiction of this first order term, does not alter the smoothness effect of the degeneracy rate, even if $m  > \gamma$. 

 Although the optimal regularity exponent is anticipated, establishing sharp gradient estimates for equations involving Hamiltonian terms is highly nontrivial. The key novelty of this work lies in developing a systematic approach to control the growth of the Hamiltonian, enabling the derivation of sharp regularity estimates in this degenerate setting.

When the operator $F$ is concave, then solutions to the homogeneous, constant coefficient equation $F(D^2 u) = 0$ are locally of class $C^{1,1}$ by Evans-Krylov Theorem \cite[Theorem 6.1]{Caff-Cabre}. In particular, we can take $\alpha_0 = 1$ in the Theorem \ref{main thm}. Therefore, we obtain the sharp regularity for the solutions of \eqref{main eq} by imposing additional condition on the operator $F$. This result reads as follows: 

\begin{corollary}\label{main cor}
    Let $u \in C(B_1)$ be a normalized viscosity solution to 
    $$
    \Phi(x, \nabla u) F(D^2 u) + h(x) |\nabla u|^{m} = f(x)
    $$
    in $B_1$, where $\Phi$, $f$, $h$, $\gamma$ and $m$ are chosen as in Theorem \ref{main thm}. If $F$ is uniformly $(\lambda, \Lambda)$-elliptic and concave fully nonlinear operator, then $u$ is locally $C^{1,\frac{1}{1+\gamma}}$. 
\end{corollary}

\section{Compactness }\label{compt resul}
In this section, we state and prove a compactness result, which will be used under the smallness regime assumption of the approximation lemma in the next section. Our proof relies on the celebrated {\it Crandall-Ishii-Lions Lemma}. 

We first establish compactness results for gradient perturbations of the equation \eqref{main eq}. Namely, equations of the form 
\begin{equation*}\label{perturbed eq}
|\nabla u + \vec{q}|^{\gamma} F(x, D^2 u ) + h(x) |\nabla u + \vec{q}|^{m} = f(x) \quad \text{in} \quad B_1,   
\end{equation*} 
where $\vec{q}$ is an arbitrary vector in $\mathbb{R}^n$. As outlined in the introduction, we consider two distinct cases. The first and more technically challenging occurs when \( \gamma < m \leq 1 + \gamma \). In this range, special attention must be given to the behavior of the equation as \( |\vec{q}| \) becomes large.  Indeed, when $ \gamma < m$, the equation can be interpreted, in the viscosity sense, as
$$
    F(x, M) + h(x) |p + \vec{q}| ^{m -\gamma} = f(x) |p + \vec{q}| ^{-\gamma}.
$$
 Therefore, in order to apply the Crandall-Ishii-Lions Lemma, it is crucial to control the growth of the Hamiltonian term $ h(x) |p + \vec{q}|^{m -\gamma}$. This difficulty motivates the compactness result that follows.

\begin{prop}\label{compacness}
    Let $\vec{q} \in \mathbb{R}^n$ and  $u \in C(B_1)$ be a normalized viscosity solution to
    \begin{eqnarray}\label{perturbed eq}
        | \nabla u + \vec{q} |^{\gamma} F(x, D^2 u ) + h(x) | \nabla u + \vec{q} |^{m} = f(x) \quad \text{in} \quad B_1, 
   \end{eqnarray}
    with $f \in L^{\infty} (B_1)$, $h \in C(B_1) \cap 
    L^{\infty}(B_1)$, $F$ is a uniformly $(\lambda, \Lambda)$-elliptic and $ \gamma < m \le 1 + \gamma$.   There exists a universal constant $C_0 > 0$ such that if
    \begin{equation}\label{smallness 101}
    \| h\|_{L^{\infty} (B_1)} \left( | \vec{q} |^{m -\gamma} + 1\right) \le C, 
    \end{equation}
    for some $C \le  C_0$, then $u$ is locally H\"older continuous in $B_{1}$ with estimates depending only upon $n, \lambda, \Lambda,  \|f \|_{L^{\infty} (B_1)}$, and $C_0$.  In particular, there exist universal constants $\alpha \in (0,1)$ and $C_* > 0$ such that 
    \begin{equation}\label{Holder contin smallness reg}
       \sup_{ \substack{x,y \, \in \, B_{1/2} \\ x \neq y }} \frac{|u(x)-u(y)|}{|x-y|^{\alpha}} \le C_* .
    \end{equation}
\end{prop}

\begin{proof}
We start off by making the following universal choices:
 \begin{equation}\label{the choices of alpha and C_0}
        0 < \alpha <\left( \frac{1}{6} \right)^{1/(m - \gamma)} \quad \text{and} \quad 0< C_0 \le  \alpha(1 - \alpha) \lambda .
    \end{equation}
Next, let  $a_0 = \left( \frac{C_0}{\| h\|_{\infty}}\right)^{1/(m-\gamma)}$. Note that by \eqref{smallness 101}, there holds 
\begin{equation}\label{choice of a_0}
     a_0 \ge 1 \quad \text{and} \quad |\vec{q}| \le a_0. 
\end{equation}
For the sake of clarity, we divide the proof into several steps. 

  \noindent{\textbf{Step 1} -}  For all $x_0 \in B_{r/2}$, $0 < r \ll 1$, we consider $\phi, \psi : \bar{B_r} \times \bar{B_r} \to \mathbb{R}$ defined by
    $$
        \psi (x, y) = L_1 |x - y|^{\alpha} + L_2 ( |x - x_0 |^2 + |y -x_0 |^2 )
    $$
    and 
    $$
        \phi (x, y) = u(x) - u(y) - \psi(x, y).
    $$
    The strategy now is to show that for some appropriate choices of $L_1, L_2 > 0$, we have
    $$
        \mathcal{L}(x_0) := \sup\limits_{(x, y) \in B_r \times B_r} \phi (x,y) \le 0.
    $$
   To prove this, we proceed by contradiction; that is, we assume, for the sake of contradiction, that for all positive $L_1$ and $L_2$, there exists $x_0 \in B_{r/2}$ such that $ \mathcal{L}(x_0) > 0.$

    Let $(\bar{x}, \bar{y} ) \in \bar{B_r} \times \bar{B_r}$ be a point where $\phi$ reaches its maximum. Thus,

     $$
        \phi(\bar{x}, \bar{y} ) = \mathcal{L} > 0.
    $$
   It is straightforward that
   \[
   \psi(\bar{x}, \bar{y}) < u(\bar{x}) - u(\bar{y}) \le 2\| u\|_{L^{\infty}(B_1)} \le 2, 
   \]
which implies
   \[
   L_1 |\bar{x} - \bar{y}|^{\alpha} + L_2 ( |\bar{x} - x_0 |^2 + |\bar{y} -x_0 |^2 ) \le 2.
   \]
By choosing $L_2 := \left( \frac{4\sqrt{2}}{r}\right)^2$, we have 
\begin{equation}\label{the points are interior}
    |\bar{x} - x_0|\le \frac{r}{4} \quad \text{and} \quad |\bar{y} - x_0|\le \frac{r}{4},
\end{equation}
\emph{i.e.}, $\bar{x}$ and $\bar{y}$ are interior points in $B_r(x_0)$. We observe that $\bar{x}\not= \bar{y}$, otherwise, $\mathcal{L} \le 0$ trivially. 

We note that $|\bar{x} - \bar{y}| \le \frac{r}{2}$, for $ 0 < r \ll 1$. Thus, hereafter, $r \in (0,1)$ is chosen to so small that 
\begin{equation}\label{bounds for the interior points}
     0 < r \le  \min\left\{ 1, \left( \frac{6C_0}{C_F \, \alpha^{1 +\gamma - m}}\right)^{1/\tau } \right\}.
\end{equation}
The above estimate shell be used in the last step. 

  \noindent{\textbf{Step 2} -} To proceed, we invoke the Crandall-Ishii-Lions Lemma, as stated in Lemma \ref{Crandall-Ishii-Lions},  which says that for every $\varepsilon > 0$, there exist matrices $X, Y \in {\rm Sym}(n)$ such that 
    \begin{equation}\label{viscosity-ineq1}
         | \xi_{\bar{x}} + \vec{q} |^{\gamma} F(\bar{x}, X ) + h(\bar{x})|\xi_{\bar{x}} + \vec{q} |^m - f(\bar{x})  \le 0 
    \end{equation}
    and 
    \begin{equation}\label{viscosity-ineq2}
        | \xi_{\bar{y}} + \vec{q} |^{\gamma} F(\bar{y}, Y ) + h(\bar{y}) | \xi_{\bar{y}} + \vec{q} |^m - f(\bar{y})  \ge 0,
    \end{equation}
   where $\xi_{\bar x}$ and $\xi_{\bar y}$ are defined by
 \[
 \xi_{\bar x}:= D_x \psi(\bar{x}, \bar{y})= L_1 \alpha |\bar{x} - \bar{y}|^{\alpha -2}(\bar{x}- \bar{y}) + 2 L_2(\bar{x} - x_0)
 \]
 and 
 \[
 \xi_{\bar y}:= - D_y \psi(\bar{x}, \bar{y}) = L_1 \alpha |\bar{x} - \bar{y}|^{\alpha -2}(\bar{x}- \bar{y}) - 2 L_2(\bar{y} - x_0).
 \]

 In addition, the matrices $X$ and $Y$ satisfy the following inequality 
 \begin{equation}\label{matrix_ineq}
\left( 
\begin{array}{cc}
X & 0 \\
0 & -Y \\
\end{array}
\right)
\leq  
\left( 
\begin{array}{cc}
Z & -Z \\
-Z & Z \\
\end{array}
\right)+ (2 L_2I + \varepsilon A^2 ), 
 \end{equation}
where $A:= D^2\psi(\bar{x}, \bar{y})$ and   $Z$ is given by
\begin{equation}\label{matrix-Z}
  Z:= L_1 \alpha|\bar{x} - \bar{y}|^{\alpha - 4} \left( |\bar{x} - \bar{y}|^2 I - (2 -   \alpha)(\bar{x} - \bar{y}) \oplus (\bar{x} - \bar{y})\right).   
\end{equation}

  \noindent{ \textbf{Step 3}} - In the sequel,  we obtain the following estimates  
   \begin{equation}\label{Jet vectors bounds}
        \frac{\alpha L_1|\bar{x} - \bar{y}|^{\alpha - 1} }{2}\le|\xi_{\bar      x}|, \,  | \xi_{\bar y}| \le  2\alpha L_1 |\bar{x} - \bar{y}|^{\alpha - 1} . 
   \end{equation}
 Indeed, by taking $L_1$ large enough so that $\displaystyle L_1 > \frac{L_2 r^{2-\alpha}}{\alpha 2^{1-\alpha}}$, we get 
 \begin{eqnarray}
     2 L_2 |\bar{x} - x_0 | &\le& 2L_2 \left( \frac{r}{4}\right) \nonumber \\
     &\le &\frac{\alpha L_1}{2} \left(\frac{r}{2}\right)^{\alpha - 1}\nonumber \\
     &\le& \frac{\alpha L_1}{2} |\bar{x} - \bar{y}|^{\alpha - 1} ,
 \end{eqnarray}
 since, by \eqref{the points are interior} 
 $$
 |\bar{x} - \bar{y}| \le |\bar{x} - x_0| + |\bar{y} - x_0| \le  \frac{r}{2},
 $$
 and $(\alpha - 1) < 0$.  Next, we estimate
 \begin{eqnarray}
     |\xi_{\bar x}| &=& \left|  L_1 \alpha |\bar{x} - \bar{y}|^{\alpha -2}(\bar{x}- \bar{y}) + 2 L_2(\bar{x} - x_0)| \right| \nonumber\\
     &\ge& L_1 \alpha |\bar{x} - \bar{y}|^{\alpha -1}  - 2L_2|\bar{x} - x_0 |\nonumber \nonumber\\
     &\ge& \frac{\alpha L_1}{2} |\bar{x} - \bar{y}|^{\alpha - 1}. \nonumber
 \end{eqnarray}
 This proves the lower bound in \eqref{Jet vectors bounds}. The upper bound follows directly from the definition of $\xi_{\bar x}$, together with the choice for $L_1$ and \eqref{the points are interior}.  The estimate for  $\xi_{\bar y}$ follows similarly. 
By further choosing $\displaystyle L_1 > \alpha^{-1} 4a_0$, and combining \eqref{choice of a_0},   \eqref{Jet vectors bounds}, and  $|\bar{x} - \bar{y}|^{\alpha - 1} > 1$, we get
\begin{eqnarray}
   | \xi_{\bar x} + \vec{q}| &\ge& | \xi_{\bar x}| - |\vec{q}| \nonumber \\
                                    & \ge& \frac{\alpha L_1}{2} |\bar{x} - \bar{y}|^{\alpha - 1} - a_0 \nonumber \\
                                    &\ge& \frac{4a_0}{2} |\bar{x} - \bar{y}|^{\alpha - 1} - a_0 \nonumber \\
                                    &\ge& \frac{4a_0}{2} |\bar{x} - \bar{y}|^{\alpha - 1} - 1 \nonumber \\
                                    &\ge &  \frac{4a_0}{2} |\bar{x} - \bar{y}|^{\alpha - 1} -  |\bar{x} - \bar{y}|^{\alpha - 1} \nonumber \\
                                    &=& a_0 |\bar{x} - \bar{y}|^{\alpha - 1} \nonumber
\end{eqnarray}

\begin{eqnarray}
     |\xi_{\bar x} + \vec{q}| &\le& |\xi_{\bar x}| + |\vec{q}| \nonumber \\
     &\le& 2\alpha L_1 |\bar{x} - \bar{y}|^{\alpha - 1}  + a_0 \nonumber \\
     &\le&  2\alpha L_1 |\bar{x} - \bar{y}|^{\alpha - 1}  + \frac{\alpha L_1}{ 4} \nonumber \\
     &\le&  2\alpha L_1 |\bar{x} - \bar{y}|^{\alpha - 1} + \alpha L_1 |\bar{x} - \bar{y}|^{\alpha - 1} \nonumber
\end{eqnarray}
Similarly, we obtain the bounds below and above to $ |\xi_{\bar y} + \vec{q}|$. Therefore,
 \begin{equation}\label{Jet vectors bounds pert}
        a_0 |\bar{x} - \bar{y}|^{\alpha - 1} \le|\xi_{\bar      x} + \vec{q}|, \,  | \xi_{\bar y} + \vec{q}| \le  3\alpha L_1 |\bar{x} - \bar{y}|^{\alpha - 1}  
   \end{equation}

 \noindent{\bf Step 4 - } Applying the matrix inequality \eqref{matrix_ineq} to vectors of the form $(v, v)\in\mathbb{R}^{2n}$ with $v \in {\mathbb S}^{n-1}$, we obtain 
\[
\langle (X-Y)v, v \rangle \leq (4L_2 + 2\varepsilon \eta),
\]
where $\eta := \|A^2\|$. Then, we conclude that all the eigenvalues of $X-Y$ are below $4L_2 + 2\varepsilon \eta$. In addition, if we apply \eqref{matrix_ineq} to the vectors of the type $(z, -z)\in\mathbb{R}^{2n}$, where
\begin{equation}\label{Estimate-Pucci-minus}
	z\,:=\,\frac{\bar{x}\,-\,\bar{y}}{|\bar{x}\,-\,\bar{y}|},
\end{equation}
we get 
\begin{equation}\label{eq3_IL}
\langle(X-Y)z, z \rangle \leq \left( - 4 \alpha (1 - \alpha) |\bar{x} - \bar{y}|^{\alpha - 2} L_1   + (4L_2 + 2\varepsilon \eta)\right)|z|^2.
\end{equation}
This implies that at least one eigenvalue of matrix $X-Y$ is below $-4  \alpha (1 - \alpha) |\bar{x} - \bar{y}|^{\alpha - 2}L_1 + 4L_2 + 2\varepsilon \eta$, which is a negative number for $L_1$ large enough. 
In the sequel, we  compute 
 \begin{eqnarray}\label{estm pucci}
        \\
        \mathcal{M}_{\lambda, \Lambda}^{-} (X - Y) \ge  \frac{4 \alpha (1 - \alpha)  \lambda L_1}{|\bar{x} - \bar{y}|^{2 - \alpha} } - (\lambda + (n-1) \Lambda)(4 L_2 + 2 \varepsilon \eta). \nonumber
 \end{eqnarray}

 Applying the ellipticity property for $F$ at the point $\bar{x} \in \mathbb{R}^n$, we obtain
 \begin{equation}\label{ineq-x}
     \mathcal{M}_{\lambda, \Lambda}^{-} (X - Y) \le F(\bar{x}, X) - F(\bar{x}, Y)
 \end{equation}
Adding the $- F(\bar{y}, Y) +  F(\bar{y}, Y)$ in the right-hand side of the inequality \eqref{ineq-x}, we get 
 \begin{eqnarray} \label{ineq:pucci-estimates}
     \mathcal{M}_{\lambda, \Lambda}^{-} (X - Y) 
     &\le & (F(\bar{x}, X) - F(\bar{x}, Y)) - F(\bar{y}, Y) +  F(\bar{y}, Y)\nonumber\\
     &\le & (F(\bar{x}, X) - F(\bar{y}, Y)) + ( F(\bar{y}, Y)- F(\bar{x}, Y))  \nonumber\\
     &\le & (F(\bar{x}, X) - F(\bar{y}, Y)) + | F(\bar{y}, Y)- F(\bar{x}, Y)| . \nonumber
 \end{eqnarray}
 Hence, it follows from the H\"older continuity of the coefficients of $F$  \eqref{uniform-continuityF} that
\begin{equation} \label{ineq_Pucci-estimates*}
\begin{array}{ccl}
   \mathcal{M}_{\lambda, \Lambda}^{-} (X - Y) 
    &\le& (F(\bar{x}, X) -   F(\bar{y}, Y) )  + C_F |\bar{x} - \bar{y}|^{\tau}(1 + \|Y\|) \\   
    &=& (\text{I}) + (\text{II}) \\
\end{array}
\end{equation}
     where $0<\tau< 1$,
     $$
        (\text{I}) = (F(\bar{x}, X) -   F(\bar{y}, Y) ) ,
    $$ 
    and 
    $$ 
        (\text{II}) =  C_F |\bar{x} - \bar{y}|^{\tau}(1 + \|Y\|). 
    $$
    In the sequel to estimate $(\text{I})$, we use the the viscosity inequalities \eqref{viscosity-ineq1},  \eqref{viscosity-ineq2},  together  with the lower bound in \eqref{Jet vectors bounds pert} to get
     \begin{eqnarray}
         (\text{I}) &\le&  \frac{f(\bar{x})}{|\xi_{\bar{x}} + \vec{q}|^{\gamma}} - \frac{f(\bar{y})}{|\xi_{\bar{y}} + \vec{q}|^{\gamma}} + h(\bar{y}) |\xi_{\bar{y}} + \vec{q}|^{m - \gamma} - h(\bar{x} )|\xi_{\bar{x}} + \vec{q}|^{m - \gamma} \nonumber \\
         &\le&2a_0^{-\gamma} \|f\|_{\infty} + \|h\|_{\infty} (|\xi_{\bar{x}}  +\vec{q}|^{m -\gamma}  + | \xi_{\bar{y}}  +\vec{q}|^{m -\gamma} ) \nonumber
     \end{eqnarray}
Moreover, since $a_0 > 1$ and  now using the upper bound in \eqref{Jet vectors bounds pert}, we have 
\begin{equation}\label{estimate for I}
    (\text{I}) \le  2 \|f \|_{\infty} + 2 \|h\|_{\infty} \left( 3\alpha L_1 |\bar{x} - \bar{y}|^{\alpha - 1} \right)^{m - \gamma} .
\end{equation} 
Now, we estimate $(\text{II})$. It follows from the definition of the matrix $Z$ in \eqref{matrix-Z} that
$$
    Z \le L_1 \alpha |\bar{x} - \bar{y} |^{\alpha -2} I,
$$
where $I$ is the identity matrix. Therefore, applying \eqref{matrix_ineq} to vectors of the form $(0, v)\in\mathbb{R}^{2n}$ with $v \in {\mathbb S}^{n-1}$, we obtain 

$$
   \langle - Y v, v \rangle \le  \langle Z v,  v \rangle + (2L_2 + \varepsilon \eta ) |v|^2 .
$$
Thus, 
\begin{eqnarray}\label{est for II}
    (\text{II}) &=&  C_F |\bar{x} - \bar{y}|^{\tau}(1 + \|Y\|) \nonumber \\
    &\le& C_F |\bar{x} - \bar{y}|^{\tau} ( 1 + L_1 \alpha |\bar{x} - \bar{y} |^{\alpha -2} + 2L_2 + \varepsilon \eta   ) \nonumber \\
    &=& C_F\alpha L_1 |\bar{y} - \bar{x}|^{\tau + \alpha -2} + C_F|\bar{x} - \bar{y}|^{\tau} ( 1 + 2L_2 + \varepsilon \eta  )
\end{eqnarray}

Therefore, adding \eqref{estimate for I} and \eqref{est for II} together, we obtain
\begin{eqnarray}\label{estimate with matrix}
     \mathcal{M}_{\lambda, \Lambda}^{-} (X-Y)& \le&    2 \|f \|_{\infty} + 2 \|h\|_{\infty} \left( 3\alpha L_1 |\bar{x} - \bar{y}|^{\alpha - 1} \right)^{m - \gamma}  \nonumber \\
     &+& C_F\alpha L_1 |\bar{y} - \bar{x}|^{\tau + \alpha -2} + C_F |\bar{y} - \bar{x}|^{\tau} ( 1 + 2L_2 + \varepsilon \eta  ) 
\end{eqnarray}

 \noindent{\textbf{Step 5} -}
Finally,  for $L_1 > 1$, \eqref{estm pucci} and \eqref{estimate with matrix} yield
\begin{eqnarray}
        4 \alpha (1 - \alpha)\lambda  L_1
        &\le& \left[2 (\lambda + (n-1) \Lambda)( 2L_2 +  \varepsilon \eta) +  \mathcal{M}_{\lambda, \Lambda}^{-} (X - Y) \right] |\bar{x} - \bar{y}|^{2 -\alpha} \nonumber \\
        &\le& \left[ c(n, \lambda, \Lambda, L_2)  + 2 \| f\|_{\infty} \right]  +  2 \|h\|_{\infty}\left( 3\alpha L_1 \right)^{m - \gamma} + \alpha C_F  L_1 |\bar{x} - \bar{y}|^{\tau} \nonumber \\
        &\le& \left[ c(n, \lambda, \Lambda, L_2)  + 2 \| f\|_{\infty} \right]  +  2 C_0 \left( 3\alpha L_1 \right)^{m - \gamma} + \alpha C_F  L_1 |\bar{x} - \bar{y}|^{\tau} \nonumber \\
        &\le& \left[ c(n, \lambda, \Lambda, L_2)  + 2 \| f\|_{\infty}\right]  + \left[  6 C_0 {\alpha }^{m - \gamma} + \alpha C_F  |\bar{x} - \bar{y}|^{\tau}  \right] L_1 . \nonumber
    \end{eqnarray}
    
 \noindent{\textbf{Claim} : }    This, together with the choices made in \eqref{the choices of alpha and C_0},   yields
    $$
        4 \alpha (1 - \alpha)\lambda  L_1 - \left[  6 C_0 {\alpha}^{m - \gamma} + \alpha C_F  |\bar{x} - \bar{y}|^{\tau}  \right] L_1  \ge 2 \alpha (1 - \alpha)\lambda L_1 ,
    $$
    Indeed, we make the following choices 
    $$
    0 < \alpha <\left( \frac{1}{6} \right)^{1/(m - \gamma)} \quad \text{and} \quad 0< C_0 \le  \alpha(1 - \alpha) \lambda .
    $$
    With these choices there holds,
    $$
        6 C_0 \alpha^{m - \gamma} \le C_0 \le \alpha(1 - \alpha) \lambda.
    $$
    Moreover, together with the fact that 
    $$
        |\bar{x} - \bar{y}|^{\tau} \le \frac{6C_0}{C_F \, \alpha^{1 +\gamma - m}}
    $$
    yields,
    \begin{eqnarray}
        \left[  6 C_0 (\alpha  )^{m - \gamma} +  \alpha C_F |\bar{x} - \bar{y}|^{\tau}  \right] &\le& \alpha (1 - \alpha )\lambda + 6C_0 \frac{C_F \, \alpha \, 6\, C_0 }{C_F \, \alpha^{1 +\gamma - m}} \nonumber \\
        &=& \alpha (1 - \alpha )\lambda + 6C_0 \alpha^{m-\gamma} \nonumber \\ 
        &\le& 2 \alpha (1 - \alpha) \lambda.  \nonumber
    \end{eqnarray}
    and the claim is proved. 
Hence, 
$$
    2 \alpha (1 - \alpha)\lambda L_1 \le  \left[ c(n, \lambda, \Lambda, L_2)  + 2 \| f\|_{L^{\infty}(B_1)} \right] 
$$
which leads to a contradiction for an $L_1$ even larger. 
\end{proof}

In the sequel, we consider \( 0 < m \leq \gamma \). As mentioned in the introduction, in this range, and away from regions where the gradient is small, that is, when \( |\vec{q}| \) is large, the equation behaves, roughly speaking, like a uniformly elliptic equation.

In this setting, we can apply a suitable adjustment of Lemmas 4 and 5 in \cite{ImbertSilvestre2013}.

\begin{prop}\label{Lip for q large}
    Let $\vec{q} \in \mathbb{R}^n$ and  $u \in C(B_1)$ be  a normalized viscosity solution to 
    \begin{eqnarray}\label{perturbed eq for m small}
        | \nabla u + \vec{q} |^{\gamma} F(x, D^2 u ) + h(x) | \nabla u + \vec{q} |^{m} = f(x) 
    \end{eqnarray}
  in  $B_1$,  where $f \in L^{\infty} (B_1)$, $h \in C(B_1) \cap 
    L^{\infty}(B_1)$, $F$ is a uniformly $(\lambda, \Lambda)$-elliptic operator, as in \eqref{unif ellipticity}. Suppose $  0 < m \le \gamma$. Then, there are a universal constants $A_0$, $\tilde{C} > 0$ such that 
    \begin{itemize}
        \item[a)]if $ |\vec{q}| \ge A_0$, then 
        \[
        \sup_{ \substack{x,y \, \in \, B_{1/2} \\ x \neq y }} \frac{|u(x)-u(y)|}{|x-y|} \le \tilde{C}.
        \]
            
        \item [b)] If $|\vec{q}| < A_0$, then  
        \[
        \sup_{ \substack{x,y \, \in \, B_{1/2} \\ x \neq y }} \frac{|u(x)-u(y)|}{|x-y|^{\delta}} \le \tilde{C}.
        \]
    \end{itemize}
    for some (universal) exponent $\delta \in (0,1)$. 

\end{prop}

\begin{remark}
   For simplicity, we can take $\tilde{C}$ to be the same constant in (a) and (b), as we can choose $\tilde{C} = \max\{ [u]_{C^{0,1}}, [u]_{C^{0,\alpha} }\}$.
\end{remark}
\begin{proof}
We observe that the proof of this proposition proceeds along the same lines as in Proposition. 
\medskip 

\noindent{\textbf{Case 1} -} We will start by proving the first case. Let $\vec{q} \in \mathbb{R}^n$ such that $|\vec{q}| \ge A_0$,  for some $A_0$ to be determined later in the proof. \ref{compacness}, and we only highlight the main differences. 

\smallskip

\noindent{\textbf{Step 1} -} For all $x_0 \in B_{r/2}$, $0 < r \ll 1$. We consider $\phi, \psi : \bar{B_r} \times \bar{B_r} \to \mathbb{R}$ defined by
    $$
        \phi (x, y) = u(x) - u(y) - \psi(x, y)
    $$
   with
    $$
        \psi (x, y) = L_1 \omega(|x - y|) + L_2 ( |x - x_0 |^2 + |y -x_0 |^2 ),
    $$
    where   $ \omega(t) = t - \kappa t^{1+\alpha}$,  with $0 < \alpha \le \tau$. We recall that $\tau$ is the Holder exponent in A\ref{oscillation condition}. And $\kappa = \kappa(\alpha) > 0$  is a constant chosen so that 
    $$
        \omega(t) \ge 0,\quad \frac{1}{2} \le \omega'(t) \le 1, \quad \omega''(t) < 0, \, \, \text{for all}\, \, \,  t \in (0,1).
    $$
 Define 
     $$
        \mathcal{L}:= \sup\limits_{\bar{B}_r \times \bar{B}_r}  \phi (x, y).
    $$
As it is usual, we proceed by contradiction. Suppose that there exists $x_0 \in B_{r/2}$ such that ${\mathcal L}>0$
for all $L_1>0$ and  $L_2>0$. We denote by $(\bar{x}, \bar{y}) \in \bar{B}_r \times \bar{B}_r$ a point where the maximum of $\phi$ is attained. Since ${\mathcal L}>0$, it is straightforward that
   \[
   L_1 \omega(|\bar{x} - \bar{y}|) + L_2 ( |\bar{x} - x_0 |^2 + |\bar{y} -x_0 |^2 ) \le 2\| u\|_{L^{\infty}(B_1)} \le 2.
   \]
By choosing $L_2 := \left( \frac{4\sqrt{2}}{r}\right)^2$, we have 
\begin{equation}\label{the points are interior version 2}
    |\bar{x} - x_0|\le \frac{r}{4} \quad \text{and} \quad |\bar{y} - x_0|\le \frac{r}{4},
\end{equation}
\emph{i.e.}, $\bar{x}$ and $\bar{y}$ are interior points in $B_r(x_0)$. We observe that $\bar{x}\not= \bar{y}$, otherwise, $\mathcal{L} \le 0$ trivially. 

\smallskip

\noindent{\textbf{Step 2} -} We again resort to the Crandall-Ishii-Lions Lemma stated in the Lemma \ref{Crandall-Ishii-Lions}. We proceed with the computation of $ D_x \psi$ and $D_y \psi$ at the point $(\bar{x}, \bar{y})$. For clarity, we set $\xi_{\bar x}$ and $\xi_{\bar y}$ as follows: 
 \[
 \xi_{\bar x}:= D_x \psi(\bar{x}, \bar{y})= L_1 \omega'( |\bar{x} - \bar{y})z + 2 L_2(\bar{x} - x_0)
 \]
 and 
 \[
 \xi_{\bar y}:= - D_y \psi(\bar{x}, \bar{y}) = L_1 \omega'( |\bar{x} - \bar{y})z - 2 L_2(\bar{y} - x_0).
 \]
where
$$
    z:= \frac{(\bar{x}- \bar{y})}{|\bar{x}- \bar{y}|}
$$

Applying Lemma \ref{Crandall-Ishii-Lions}, we have that for every $\varepsilon > 0$, there exist matrices $X, Y \in {\rm Sym}(n)$ such that 
    \begin{equation}\label{eq sup}
         | \xi_x + \vec{q} |^{\gamma} F(\bar{x}, X ) + h(\bar{x})|\xi_x + \vec{q} |^m - f(\bar{x})  \le 0 
    \end{equation}
    and 
    \begin{equation}\label{eq sub}
         | \xi_y + \vec{q} |^{\gamma} F(\bar{y}, Y ) + h(\bar{y}) | \xi_y + \vec{q} |^m - f(\bar{y})  \ge 0 .
    \end{equation}
  In addition, 
 \begin{equation}\label{matrix_ineq version 2}
\left( 
\begin{array}{cc}
X & 0 \\
0 & -Y \\
\end{array}
\right)
\leq  
\left( 
\begin{array}{cc}
Z & -Z \\
-Z & Z \\
\end{array}
\right)+ (2 L_2I + \varepsilon A^2 ), 
 \end{equation}
 where $A^2:= D^2\psi(\bar{x}, \bar{y})$ and   
 \begin{equation}\label{Z part 2}
     Z:= L_1 {\omega}^{\prime \prime}(|\bar{x} - \bar{y}|) \dfrac{(\bar{x} - \bar{y})\otimes(\bar{x} - \bar{y}) }{|\bar{x} - \bar{y}|^2} + L_1\dfrac{{\omega}^{\prime}(|\bar{x} - \bar{y}|)}{|\bar{x} - \bar{y}|}\left( I - \dfrac{(\bar{x} - \bar{y})\otimes(\bar{x} - \bar{y}) }{|\bar{x} - \bar{y}|^2}\right).
 \end{equation}
\smallskip

\noindent{\textbf{Step 3} -} Next, we apply the matrix inequality \eqref{matrix_ineq version 2} to vectors of the form $(v,v)$, where $v \in \mathbb{S}^{n -1}$ to obtain
$$
    \langle (X - Y) v, v \rangle \le (4L_2 + 2\varepsilon \eta)
$$
where $\eta:= \| A^2\|$, which implies all the eigenvalues of $X-Y$ are below $4L_2 + 2\varepsilon \eta$. On the other hand, we employ the same inequality \eqref{matrix_ineq version 2} to the particular vector $(z, -z)$, with
$$
    z:= \frac{(\bar{x}- \bar{y})}{|\bar{x}- \bar{y}|}
$$
to get
\begin{eqnarray}
    \langle (X - Y) z, z\rangle &\leq& (4L_2 + \varepsilon \eta + 4L_1 \omega'' (|\bar{x} - \bar{y}|) |z|^2 \nonumber \\
\end{eqnarray}
where 
$\omega''(t) =  - \kappa (1+\alpha)\alpha t^{\alpha -1}$ . As before, this implies that at least one eigenvalue of $X - Y$ is less than $4L_2 + \varepsilon \eta + 4L_1 \omega'' (|\bar{x} - \bar{y}|)$, which will be negative for $L_1$ large enough. Therefore, by the definition of the Pucci operators, 
\begin{eqnarray}\label{pucci ineq}
    \mathcal{M}_{\lambda, \Lambda}^{-} (X - Y) \ge (n-1)\Lambda (4L_2 + 2\varepsilon \eta) - \lambda (4L_2 + \varepsilon \eta + 4L_1 \omega'' (|\bar{x} - \bar{y}|) ) . \nonumber
\end{eqnarray}
\smallskip

\noindent{\textbf{Step 4} -}
Recall  \eqref{ineq_Pucci-estimates*}:
\begin{eqnarray}\label{eq_Pucci-estimates part 2}
   \mathcal{M}_{\lambda, \Lambda}^{-} (X-Y) &\le&  (F(\bar{x}, X) - F(\bar{y}, Y)) +  C_F |\bar{x} - \bar{y}|^{\tau}(1 + \|Y\|) \nonumber \\ 
   &=& \tilde{\rm I} + \tilde{\rm II},
   \end{eqnarray}
   where 
   $$
   \tilde{\rm I}:=  (F(\bar{x}, X) - F(\bar{y}, Y))
   $$
   and
   $$
   \tilde{\rm   II}:= C_F |\bar{x} - \bar{y}|^{\tau}(1 + \|Y\|). 
   $$

To estimate $\tilde{\rm I}$, we recall that $|\vec{q}| \ge A_0$. Then since $|\xi_{\bar{x}}| \le 2L_1$, for $L_1 > L_2$, and taking $A_0 \ge 50 L_1$ we obtain 
    $$
        |\vec{q} + \xi_{\bar{x}}| \ge |\vec{q}| - |\xi_{\bar{x}}| \ge 50 L_1 - 2 L_1 = 48 L_1,
    $$
 A similar computation yields,
$$
    |\xi_{\bar{y}} + \vec{q} | \ge 48L_1. 
$$

The above inequalities combined with \eqref{eq sup}, \eqref{eq sub} yield for $L_1 \gg 1$,
   \begin{eqnarray} \label{ineq: F part 1}
   \tilde{\rm I} 
   &=&    F(\bar{x}, X) - F(\bar{y}, Y) \nonumber\\
   &\le& \frac{2\| f\|_{\infty}}{(48 L_1)^{\gamma}} + \frac{2\|h\|_{\infty}}{(48 L_1)^{\gamma - m}} \\
    &\le& 2 ( \|f\|_{\infty} + \|h\|_{\infty} ).\nonumber
\end{eqnarray}

Finally, we estimate $\tilde{\rm II}$. It follows from the definition of the matrix $Z$ in \eqref{Z part 2} and the fact that $\omega'\le 1$ and $\omega''(t) \le 0$, that
$$
    Z \le L_1 \frac{ \omega'( |\bar{x} - \bar{y} |)}{  |\bar{x} - \bar{y} |} {\rm I} \le L_1   |\bar{x} - \bar{y} |^{-1} { \rm I},
$$
where I is the identity matrix. Therefore, applying \eqref{matrix_ineq version 2} to vectors of the form $(0, v)\in\mathbb{R}^{2n}$ with $v \in {\mathbb S}^{n-1}$, we obtain 
$$
   \langle - Y v, v \rangle \le  \langle Z v,  v \rangle + (2L_2 + \varepsilon \eta ) |v|^2 .
$$
Thus,
\begin{eqnarray}\label{ineq: F part 2}
    \tilde{\text{II}} &=&  C_F |\bar{x} - \bar{y}|^{\tau}(1 + \|Y\|) \nonumber \\
    &\le& C_F |\bar{x} - \bar{y}|^{\tau} ( 1 + L_1  |\bar{x} - \bar{y} |^{-1} + 2L_2 + \varepsilon \eta   ) \nonumber \\
    &=& C_F L_1 |\bar{y} - \bar{x}|^{\tau - 1} + C_F|\bar{x} - \bar{y}|^{\tau} ( 1 + 2L_2 + \varepsilon \eta  )
\end{eqnarray}
Moreover, by ellipticity, 
\begin{eqnarray}
     \mathcal{M}_{\lambda, \Lambda}^{-} (X - Y)  &\ge& - \lambda \left(4L_2 + \varepsilon \eta + 4L_1 \omega'' (|\bar{x} - \bar{y}|)\right)   + (n-1) \Lambda (4L_2 + 2\varepsilon \eta) \nonumber \\
    &=&  c(\alpha) \lambda \kappa L_1|\bar{x} - \bar{y}|^{\alpha - 1} + [  (n-1)\Lambda - \lambda ] ( 4L_2 + \varepsilon \eta ). \nonumber
\end{eqnarray}
In view of \eqref{eq_Pucci-estimates part 2}, we obtain
$$
    c(\alpha) \kappa L_1 |\bar{x} - \bar{y}|^{\alpha - 1} \le \text{\rm I} + \text{\rm II} +  C(n, \lambda, \Lambda ).
$$
Plugging \eqref{ineq: F part 1} and \eqref{ineq: F part 2} in the previous inequality, we have
\begin{eqnarray}
    c(\alpha) \kappa L_1  &\le&  C (n, \lambda, \Lambda,  \|f\|_{\infty},  \|h\|_{\infty}, L_2, C_F) + C_F L_1 |\bar{x} - \bar{y}|^{\tau - 1 +1 - \alpha } \nonumber \\ 
    &=& C (n, \lambda, \Lambda,  \|f\|_{\infty},  \|h\|_{\infty}, L_2, C_F) + C_F L_1 |\bar{x} - \bar{y}|^{\tau - \alpha} .
\end{eqnarray}
Recall that $\alpha \le  \tau$. Hence,
$$
    L_1  \le  C (n, \lambda, \Lambda,  \|f\|_{\infty},  \|h\|_{\infty}, L_2, C_F)
$$

And thus, for an even bigger $L_1$, we get a contradiction.  
\medskip

\noindent{\textbf{Case 2} -} To complete the proof, we address the remaining case. Let  $\vec{q} \in \mathbb{R}^n$ such that $|\vec{q}| < A_0$, with $A_0 = 50 L_1$. 
Consider the operator
$$
    G (x,  p , M) = | \vec{q} + p |^{\gamma} F(x,M) + h(x) |\vec{q} + p |^m  .
$$
 Thus, equation \eqref{perturbed eq for m small} can be rewritten as $G(x, Du, D^2 u) = f(x)$. Thus in particular, if $|p| \ge 5A_0$, then $| p + \vec{q} | \ge 4A_0 \ge 1$ Therefore,  $G(x, p, M) = f(x)$ with $|p| \ge 5 A_0$ implies
$$
    \left\{\begin{matrix}
\mathcal{M}^{+} (D^2 u)  + |h| + |f| \ge 0  \\
 \mathcal{M}^{-} (D^2 u) - (| h| + |f|) \le 0  \\
\end{matrix}\right. 
$$
where $\mathcal{M}^{\pm}$ are the extremal Pucci operators associated with the ellipticity constants of $F$, therefore, it is known from \cite{Delarue2010} and \cite{Imbert-Silvestre2016}, that $u$ is H\"older continuous with estimates depending only on the dimension, ellipticity constants, $\|h\|_{L^{\infty}(B_1)}$, $\|f\|_{L^{\infty}(B_1)}$ and $A_0$.

Together with the previous case, this completes the proof of the theorem. 
\end{proof}

\section{Approximations}\label{sct approx}
An essential tool towards optimal gradient regularity is the following lemma. The goal is to approximate solutions to first order perturbations of equation \eqref{main eq} to solutions of the constant coefficient, homogeneous model $F(D^2 z) = 0$. The method used here relies on stability-like and compactness results. See e.g., \cite[Lemma 4.1]{SN}, \cite[Proposition 6]{APPT}. 

\begin{lem}[Approximation Lemma]\label{Approx lemma}
Let $\vec{q} \in \mathbb{R}^n$ and $u \in C(B_1)$ be a normalized viscosity solution to 
\begin{equation}\tag{$E_{\vec{q}}$}
    | \nabla u + \vec{q} |^{\gamma} F(x, D^2 u ) + h(x)| \nabla u + \vec{q} |^m = f(x) \quad \text{in} \quad B_1 . 
\end{equation}
  Given $\delta > 0$ there exists $\eta = \eta(\delta, n, \lambda, \Lambda) > 0$, such that if 
\begin{equation}\label{smallness regime}
 \| {\rm osc}_F \|_{L^{\infty}(B_1)}  +   \|f\|_{L^{\infty}(B_1)} + \| h\|_{\infty} ( |\vec{q}|^{(m - \gamma)_+} + 1) < \eta,
\end{equation}
then there exists a function $z :B_{3/4} \to \mathbb{R}$ and a $(\lambda, \Lambda)$-  elliptic, constant coefficient operator $F$ such that 
\begin{equation}\label{eq of approx lemma}
     F (0, D^2 z) = 0 \quad 
    \text{in}\quad  B_{3/4}, 
\end{equation} in the viscosity sense, and 
$$
\| u - z \|_{L^{\infty}(B_{1/2})} < \delta. 
$$
\end{lem}

\begin{proof}
Suppose, by contradiction, that there exists $\delta_0 > 0$ and  sequences of functions  $( u_j)_{j \in \mathbb{N}}$,  $ ( h_j )_{j \in \mathbb{N}}$, $( f_j )_{j \in \mathbb{N}}$, and a sequence of $(\lambda, \Lambda)$-elliptic operators $F_j : B_1 \times {\rm Sym}(n) \to \mathbb{R}$ such that

    \begin{equation}\tag{i}
         -1 \le u_j \le 1 
    \end{equation}

     \begin{equation}\label{sequence of equations}\tag{ii}
         | \nabla u_j + \Vec{q}_j|^{\gamma} F_{j} (x, D^2 u_j ) + h_{j}(x) | \nabla u_j + \vec{q}_j |^m = f_j(x)  \quad \text{in} \quad  B_1
     \end{equation}
      and 

    \begin{equation}\label{smallness of sequence}\tag{iii}
        \displaystyle   \| \mbox{osc}_{F_j} \|_{L^{\infty}(B_1)}  +   \|f_j\|_{L^{\infty}(B_1)} + \|h_j\|_{\infty} ( |\vec{q_j}|^{(m-\gamma)_+} + 1) < \frac{1}{j}.
    \end{equation}
However, 
    \begin{equation}\label{Appr Lem contrad eq}
              \| u_j - z\|_{L^{\infty} (B_{1/2})} \ge \delta_0 > 0, 
    \end{equation} 
   for any $z$ satisfying a constant coefficient, homogeneous, $(\lambda, \Lambda)$-uniform elliptic equation \eqref{eq of approx lemma}. 
   
Initially, we note that from \eqref{smallness of sequence}, 
$$
    \|h_j\|_{\infty}  ( |\vec{q_j}|^{(m-\gamma)_+}  + 1) < j^{-1} ,
$$ for all $j \ge 1$. Therefore, by the compactness results (cf. Proposition~\ref{compacness} and Proposition~\ref{Lip for q large}, depending on $m$ and $\gamma$),  the sequence $( u_j)_{j \in  \mathbb{N}}$ is  pre-compact in the $C_{loc}^{0} (B_1 )$-topology. Indeed, if $ \gamma < m \le 1 + \gamma$, we can use Proposition \ref{compacness}. And if $m \le \gamma$, we make use of Proposition \ref{Lip for q large}. In either case, the sequence $( u_j)_{j \in  \mathbb{N}}$, is locally equi-continuous in $B_1$.

By the uniform ellipticity assumption on the operator $F$ and property \eqref{smallness of sequence}, we may assume (after passing to a subsequence if necessary) that
\[
F_j(x, M) \to F_\infty(M) := F(0, M) \quad \text{locally uniformly in } B_1 \times \mathrm{Sym}(n),
\]
and  \( F_\infty \) is a constant coefficient operator with the same ellipticity constants,  \( (\lambda, \Lambda) \).

In what follows, we want to show that 
$$
    F_{\infty}( 0,  D^2 u_{\infty} ) = 0 \, \, \text{in} \, \, B_{9/10}
$$
in the viscosity sense. Given $y \in B_{9/10}$, define 
$$
p(x): = u_{\infty} (y) + \vec{b} \cdot (x - y) + \frac{1}{2} (x -y )^{T} M (x-y).
$$
Without loss of generality, we assume that $p(x)$ touches $u_{\infty}$ from bellow at $y$ in $B_{3/4}$. i.e., $p(x) \le u_{\infty} (x)$, $x \in B_{3/4}$, and clearly $p(y) = u_{\infty}(y)$.  Also without loss of generality, we may assume  $y = 0$. Next, we aim to verify that 
\begin{equation}\label{goal equation}
    F_{\infty} (0, M) \le  0. 
\end{equation}
For $ 0 < r \ll 1$ fixed, we define the sequence $(x_j)_{j \in\mathbb{N}}$ as follows:
$$
    p(x_j) - u_j(x_j) = \max\limits_{x \in B_r} \left( p(x) - u_j(x_j) \right) 
$$
From \eqref{sequence of equations},  we infer that 

\begin{equation} \label{eq:limit_sequences}
    \left| \vec{b} + \vec{q_j} \right|^{\gamma} F_j(x_j, M) + h_{j} (x_j)\left| \vec{b} + \vec{q}_j \right|^m \le f_j(x_j) . 
\end{equation}
To easy the presentation of the proof, hereafter we write 
$$
    H_j (x_j, \vec{b} + \vec{q_j} )  = h_j(x_j)\left| \vec{b} + \vec{q}_j\right|^m.
$$
From here, we consider several cases:
\smallskip

\noindent{\textbf{Case 1 -}} If $(\vec{q}_j)_{j \in \mathbb{N}}$ is an unbounded sequence, then consider the (renamed) subsequence satisfying $| \vec{q}_j | > j$, for every $j \in \mathbb{N}$. Moreover, there exists $j^{*} \in \mathbb{N}$ such that 
$$
    \max\{1, |\vec{b}| \}< |\vec{q_j}|
$$
for every $j > j^{*}$. Thus, dividing the inequality \eqref{eq:limit_sequences} by $|\vec{q_j}|^{\gamma}$, and defining $\vec{e_j} = \vec{q}_j / |\vec{q}_j |$, we have 
$$
    \left| \frac{\vec{b}}{|\vec{q_j}|} + \vec{e_j} \right|^{\gamma} F_j(x_j, M) + \tilde{H}_j (x_j, \vec{b} + \vec{q_j} ) \le \frac{f_j(x_j)}{| \vec{q}_j|^{\gamma}}
$$
where $\tilde{H}_j (x_j, \vec{b} + \vec{q_j} ): =| \vec{q_j}|^{ - \gamma} H_j (x_j, \vec{b} + \vec{q_j} )  $. Note that 
\begin{eqnarray}
    \left| \Tilde{H}_j (x_j, \vec{b} + \vec{q_j} ) \right| & \le &  \| h_j\|_{\infty} | \vec{q}_j |^{m-\gamma} \left| \frac{\vec{b}}{| \vec{q}_j|} + \vec{e_j} \right|^{m}  \nonumber \\
    &\le & \| h_j\|_{\infty} | \vec{q}_j |^{m-\gamma} \left( \frac{|\vec{b}|}{| \vec{q}_j|} + 1 \right)^{m}  \nonumber \\
    &\le& 2^m \frac{|\vec{q}_j|^{m-\gamma}}{ j \left( | \vec{q}_j |^{(m-\gamma )_+} + 1   \right) } ,  \nonumber
\end{eqnarray}
by using \eqref{smallness of sequence}.  Next, we notice that if $ 0 < m - \gamma$, then 
$$
    2^m \frac{|\vec{q}_j|^{m-\gamma}}{j \left(| \vec{q}_j |^{m-\gamma} + 1   \right) } \le  \frac{2^m}{j} \to 0, \quad\text{as} \quad j \to \infty. 
$$
And, if $ m - \gamma \le 0$
$$
    2^m \frac{|\vec{q}_j|^{m-\gamma}}{j\left( | \vec{q}_j |^{(m-\gamma)_+} + 1   \right) } =   \frac{2^m}{2j|\vec{q}_j|^{\gamma - m}} \to 0, \quad\text{as} \quad j \to \infty. 
$$
Hence, in either case,  by letting $j \to \infty$ we have $\tilde{H}_j \to 0$ and hence 
$$
    | 0 + \vec{e}_{\infty} |^{\gamma} F_{\infty}( 0, M) + 0 \le 0.
$$
Thus, by the cutting Lemma in \cite[Lemma 6]{ImbertSilvestre2013}, we conclude that $F_{\infty}( M) \le 0$, as desired. 

\smallskip

\textbf{Case 2 -} Conversely, if $(\vec{q_j})_{j \in \mathbb{N}}$ is bounded, then at least through a subsequence
$$
     \vec{b} + \vec{q_{j}} \to \vec{b} + \vec{q}_{\infty}.
$$
First, we notice that by \eqref{smallness of sequence}
\begin{eqnarray}\label{hamiltonian-convergent case}
    \left| H_{j} (x_j, \vec{b} + \vec{q}_{j} )\right| &\le& \| h_j\|_{\infty} | \vec{b} + \vec{q_j} |^{m} \nonumber \\
    &\le& \left\{\begin{matrix} 
    \frac{| \vec{b} + \vec{q_j} |^{m}}{j (|\vec{q_j}|^{m-\gamma} + 1  )}, & \text{if} \,\, m - \gamma >0 \\ 

\\
    \frac{| \vec{b} + \vec{q_j} |^{m}}{2j}, & \text{if} \, \, m - \gamma  \le 0.
    \end{matrix}\right. \nonumber
\end{eqnarray}
Thus, $H_j (x_j, \vec{b} + \vec{q_j} ) \to 0$.
 For the degenerate part of the inequality \eqref{eq:limit_sequences}, we distinguish two cases:  

\smallskip

\noindent{\textbf{Case 2.1 - }}If $| \vec{b} + \vec{q}_{\infty}| \neq \vec{0}$, then 
$$
    |\vec{b} + \vec{q}_{\infty} |^{\gamma} F_{\infty} (0, M) +  0 \le 0
$$
 as $j \to \infty$, and once again, by the cutting Lemma in \cite[Lemma 6]{ImbertSilvestre2013}, we conclude  $F_{\infty} (0, M) \le 0$. 

\smallskip

\noindent{\textbf{Case 2.2 -}} It remains to analyze the case $ |\vec{b} + \vec{q}_{\infty} |  = 0$.

In this case, we first notice that if ${\rm Spec}( M ) \subset ( - \infty, 0 ]$, then ellipticity yields the inequality \eqref{goal equation} as desired. Indeed, by ellipticity
$$
    F_{\infty} (0, M) \le \mathcal{M}_{\lambda, \Lambda}^{+} (M) = \lambda \sum\limits_{i = 1}^{n} \tau_i \le 0
$$
where $\left\{\tau_i, \,\,  i=1, \dots, n\right\}$ are the eigenvalues of $M.$
Thus, we may assume that $M$ has $k\ge1 $ strictly positive eigenvalues.  Let $( \vec{e_i} )_{1\le i\le k}$ be the associate eigenvectors and define 
$$
    E = \text{Span}\{ \vec{e_1} , \dots, \vec{e_k} \}. 
$$
Consider the orthogonal sum $\mathbb{R}^{n} : = E \oplus G$ and the orthogonal projection $P_{E}$ on $E$. Define the test function
\begin{equation} \label{Test_function-AL}
    \varphi(x) = \kappa \sup_{ \vec{e} \, \in \,  \mathbb{S}^{n-1}} \langle P_E x, \vec{e} \rangle +  b \cdot x  + \frac{1}{2} x^T M x. 
\end{equation}
Because $u_j \to u_{\infty}$ locally uniformly, and then $\varphi$ touches $u_{j}$ from below at some point $x_j^{\kappa} \in B_r$.

In what follows, we analyze the two cases: $x_j^{\kappa} \in G$ and $x_j^{\kappa} \notin G$. First, we suppose $x_j^{\kappa} \in G$, then we may rewrite \eqref{Test_function-AL} as
\[
\kappa \langle P_E x, \vec{e} \rangle + b \cdot x + \frac{1}{2} x^T M x
\]
touches $u_j$ at $x_j^{\kappa}$ regardless of the direction of $\vec{e} \in \mathbb{S}^{n-1}$. Moreover,
\[
D\langle \vec{e},  P_E x\rangle  = P_E (\vec{e}) \quad \text{and} \quad D^2\langle \vec{e},  P_E x\rangle = 0.
\]
and,
\[
\vec{e} \in \mathbb{S}^{n-1} \cap E \Longrightarrow  P_E (\vec{e}) = \vec{e} \quad \text{and} \quad \vec{e} \in \mathbb{S}^{n-1} \cap G \Longrightarrow  P_E (\vec{e}) =0.
\]
By applying the test function above in the inequality \eqref{sequence of equations}, we obtain
$$
   | Mx_j^{\kappa} + \vec{b} + \vec{q_j} + \kappa  P_E (\vec{e} )|^{\gamma} F_j (x_j^{\kappa}, M) + H_j(x_j^{\kappa}, Mx_j^{\kappa} +  \vec{b} + \vec{q_j} + \kappa P_E (\vec{e} ) )\le f_j(x_j^{\kappa}) 
$$
for every $\vec{e} \in \mathbb{S}^{n-1}$. 

 In what follows, we may assume that $x_j^{\kappa} \to x_*$ for some $x_* \in B_1$ since $ \sup\limits_{j \in \mathbb{N}} |x_j^{\kappa} | \ll 1 $.  Now we consider the cases $x_* = 0$ and $x_* \not=0$.
 
 If $x_* = 0$ then the sequences $|M x_j^{\kappa} |$ and $ |\vec{b} + \vec{q}_j|$ converge to zero, we have $|M x_j^{\kappa} | + |\vec{b} + \vec{q}_j| \le \frac{\kappa}{2}$ for $j \gg 1$. Thus, for $\vec{e} \in \mathbb{S}^{n-1} \cap E $,
\begin{equation}\label{seq goes to zero}
    |M x_j^{\kappa}  + \vec{b} + \vec{q}_j + \kappa \vec{e} | \ge \frac{\kappa}{2}.
\end{equation}
Therefore, 
\begin{equation}\label{2.2.2 a}
    F_j (x_j^{\kappa} , M) + \left(\frac{\kappa}{2}\right)^{-\gamma} H_j(x_j^{\kappa} ,   Mx_j^{\kappa} + \vec{b} + \vec{q}_j + \kappa \vec{e} ) \le \left(\frac{\kappa}{2}\right)^{-\gamma} f_j(x_j^{\kappa}) .
\end{equation}
Furthermore, we use \eqref{smallness of sequence} to estimate the following quantity
\begin{eqnarray}\label{H goes to zero 2 case}
     \left|  H_j(x_j^{\kappa} ,   Mx_j^{\kappa} +\vec{b} + \vec{q}_j + \kappa \vec{e} ) \right| &\le&  \|h_j\|_{\infty} |   Mx_j^{\kappa} +  \vec{b} + \vec{q}_j +\kappa \vec{e} |^m  \nonumber \\
     &\le &  \left(   \kappa + \frac{\kappa}{2} \right)^m   \|h_j\|_{\infty}  \nonumber \\
     &\le& \left( \frac{3}{2} \kappa \right)^m \frac{1}{j ( 1 + |\vec{q_j}|^{(m-\gamma)_+} )}. 
\end{eqnarray}
Thus, we can conclude that 
\[
\left|  H_j(x_j^{\kappa} ,   Mx_j^{\kappa} +\vec{b} + \vec{q}_j + \kappa \vec{e} ) \right| \rightarrow 0 
\]
as $j \to \infty$. Finally, letting $j \to \infty$ in \eqref{2.2.2 a}, yields $F_{\infty} (0, M)  \le 0.$

Next, if $x_* \neq 0$, then $|Mx_* | > 0$. In this case, since $x^{\kappa}_j \rightarrow x_*$ and $x_j^{\kappa} \in G$, then $E \not \equiv \mathbb{R}^n$. Thus for  $\vec{e} \in \mathbb{S}^{n-1} \cap G $ we have 
$$
    |M x_* + \kappa P_E( \vec{e})| = |M x_*| >   0.
$$
Hence, by the viscosity inequalities \eqref{sequence of equations}, we have 
$$
    |Mx_j^{\kappa}  + \vec{b} + \vec{q}_j |^{\gamma} F_j(x_j^{\kappa}, M ) +  H_j(x_j^{\kappa}, Mx_j^{\kappa} +  \vec{b} + \vec{q_j} ) \le f_j(x_j^{\kappa}) . 
$$
Moreover, since 
$$
|Mx_j^{\kappa} + \kappa P_E (\vec{e} )| = |Mx_j^{\kappa}|  \to |M x_* | \quad \text{and} \quad  |\vec{b} + \vec{q}_j| \to 0
$$
as $j \to \infty$, there exists $\bar{j} >1$ such that for all $j \ge \bar{j}$ there holds, 
\begin{equation}
    |Mx_j^{\kappa} |  \ge \frac{1}{2} |M x_*  |  \quad \text{and} \quad    |\vec{b} + \vec{q}_j| \le \frac{1}{4}  |M x_* | . 
\end{equation}
Therefore, 
\begin{equation}
    |M x_j^{\kappa}  + \vec{b} + \vec{q}_j  | \ge \frac{1}{4}   |M x_*  |  > 0. 
\end{equation}
And thus, 
\begin{equation}\label{2.2.2.b}
    F_j(x_j^{\kappa}, M ) + \frac{H_j(x_j^{\kappa}, Mx_j^{\kappa} +  \vec{b} + \vec{q_j} )}{ | Mx_* |^{\gamma}}  \le \frac{f_j(x_j^{\kappa})}{|M x_* |^{\gamma}}. 
\end{equation}
A computation similar to \eqref{H goes to zero 2 case} yields
\begin{eqnarray}
    \left| \frac{H_j(x_j^{\kappa}, Mx_j^{\kappa} +  \vec{b} + \vec{q_j} )}{ | Mx_* |^{\gamma}} \right| &\le& \frac{\| h\|_{\infty} |Mx_j^{\kappa} + \vec{b} + \vec{q_j} |^m}{| Mx_* |^{\gamma} } \nonumber \\
    &\le& \frac{ |Mx_j^{\kappa} + \vec{b} + \vec{q_j} |^m}{| Mx_* |^{\gamma} \, j \, ( 1 + |\vec{q_j}|^{(m-\gamma)_+} )}  \to 0
\end{eqnarray}
as $j \to \infty. $ Finally, by letting $j \to \infty$ in \eqref{2.2.2.b} as before, we have $F_{\infty} (M) \le 0$. 

To complete the proof, we study the case $x_j^{\kappa} \notin G$, i.e., $P_E x_j^{\kappa} \neq 0$. In this case, by taking 
\[
 e_j^{\kappa} = \frac{P_E x_j^{\kappa} }{|P_E x_j^{\kappa} |} ,
\]
we have 
$$
    \sup_{ \vec{e} \in \mathbb{S}^{n-1}} \langle P_E x, \vec{e} \rangle  = |P_E x_j^{\kappa} | .
$$
Hence, $|P_E(x)|$ is infinitely differentiable at $x^{\kappa}_j$ and has gradient $e^{\kappa}_j$ and Hessian $I - e^{\kappa}_j \otimes e^{\kappa}_j$ at $x^{\kappa}_j$.

Since $u_j$ satisfies \eqref{sequence of equations} in the viscosity sense,  we obtain
\begin{eqnarray}
   f_j(x_j^{\kappa} ) &\ge& \left| Mx_j^{\kappa} + \vec{b} + \vec{q}_j + \kappa \frac{P_E x_j^{\kappa}}{|P_E x_j^{\kappa}|} \right|^{\gamma} F_j \left( x_j^{\kappa} , M + \kappa \left( I - \frac{P_E x_j^{\kappa}}{|P_E x_j^{\kappa}|} \otimes \frac{P_E x_j^{\kappa}}{|P_E x_j^{\kappa}|}\right) \right) \nonumber \\
   & & \quad   + \quad  H_j \left( x_j^{\kappa} , Mx_j^{\kappa} + \vec{b} + \vec{q}_j + \kappa \frac{P_E x_j^{\kappa}}{|P_E x_j^{\kappa}|} \right) . \nonumber
\end{eqnarray}
We write $x_j^{\kappa}$ as 
$$
    x_j^{\kappa} = \sum\limits_{i=1}^{n} a_i \vec{e_i}
$$
where $\{ \vec{e_i}, i = 1, \dots, n \}$ are the eigenvectors of $M$. Hence, 
$$
    M x_j^{\kappa} = \sum\limits_{i=1}^{k} \tau_i a_i \vec{e_i} + \sum\limits_{i=k+1}^{n} \tau_i a_i \vec{e_i},
$$
with $\tau_{i} > 0$ for $i = 1, \dots, k.$
\begin{eqnarray}
    \kappa &\le& \kappa + \frac{1}{|P_E (x_j^{\kappa})|}  \sum\limits_{i=1}^{k} \tau_i a_i^2 \nonumber \\
    &\le& \kappa +  \frac{1}{|P_E (x_j^{\kappa})|} \left\langle \sum\limits_{\ell=1}^{n} \tau_{\ell} a_{\ell} \vec{e_{\ell}}  , \sum\limits_{i=1}^{k}  a_i \vec{e_i}  \right\rangle \nonumber \\  
    &=& \kappa + \frac{1}{|P_E (x_j^{\kappa})|} \left \langle M x_j^{\kappa} , P_E(x_j^{\kappa} ) \right \rangle \nonumber \\
    &= & \left\langle M x_j^{\kappa} + \kappa \frac{P_E (x_j^{\kappa})}{|P_E (x_j^{\kappa})|} , \frac{P_E (x_j^{\kappa})}{|P_E (x_j^{\kappa})|} \right\rangle \nonumber \\
    &\le& \left| M x_j^{\kappa} + \kappa \frac{P_E (x_j^{\kappa})}{|P_E (x_j^{\kappa})|} \right|.  \nonumber
\end{eqnarray}
Therefore, since 
$$
    \kappa \left( I - \frac{P_E x_j^{\kappa}}{|P_E x_j^{\kappa}|} \otimes \frac{P_E x_j^{\kappa}}{|P_E x_j^{\kappa}|}  \right) \ge 0
$$
in the sense of matrices. It follows from the ellipticity of $F$ that
$$
    F_{j}(x_j^{\kappa}, M) \le F_j \left( x_j^{\kappa} , M + \kappa \left( I - \frac{P_E x_j^{\kappa}}{|P_E x_j^{\kappa}|} \otimes \frac{P_E x_j^{\kappa}}{|P_E x_j^{\kappa}|}\right) \right),
$$
and hence
\begin{equation}\label{ineq:limit_fully_nonlinear}
F_j( x_j^{\kappa}, M) + \left(\frac{1}{\kappa}\right)^{\gamma}  H_j \left( x_j^{\kappa} , Mx_j^{\kappa} +   \vec{b} + \vec{q}_j +\kappa \frac{P_E x_j^{\kappa}}{|P_E x_j^{\kappa}|} \right) \le \left(\frac{1}{\kappa}\right)^{\gamma} f_j(x_j^{\kappa} ).
\end{equation}
We again use \eqref{smallness of sequence} to estimate the following quantity
\begin{eqnarray}
     \left|  H_j\left(x_j^{\kappa} ,   Mx_j^{\kappa} +\vec{b} + \vec{q}_j + \kappa \frac{P_E x_j^{\kappa}}{|P_E x_j^{\kappa}|} \right) \right| &\le&  \|h_j\|_{\infty} \left|   Mx_j^{\kappa} +  \vec{b} + \vec{q}_j +\kappa \frac{P_E x_j^{\kappa}}{|P_E x_j^{\kappa}|} \right|^m  \nonumber \\
     &\le &  \left(   \kappa + \frac{\kappa}{2} \right)^m   \|h_j\|_{\infty}  \nonumber \\
     &\le& \left( \frac{3}{2} \kappa \right)^m \frac{1}{j ( 1 + |\vec{q_j}|^{(m-\gamma)_+} )}. 
\end{eqnarray}
Hence, taking the limit in the inequality above, we obtain:
\begin{eqnarray}
  \left| H_j \left( x_j^{\kappa} , Mx_j^{\kappa} + \vec{b} + \vec{q}_j+  \kappa \frac{P_E x_j^{\kappa}}{|P_E x_j^{\kappa}|} \right) \right| \rightarrow 0
\end{eqnarray}
as $j \to \infty$.  Finally, passing to the limit in \eqref{ineq:limit_fully_nonlinear} as $j \to \infty$, we get $F_{\infty} (0, M) \le 0$. Thus, $u_{\infty}$ is a supersolution of $F_{\infty}(0, D^2 u_{\infty}) = 0$, in the viscosity sense. The reasoning for showing that $u_{\infty}$ is also a viscosity subsolution is analogous, and we omit the details here. Therefore, $u_{\infty}$ solves $F(0, D^2 z) =0$.

Taking $z = u_{\infty}$ leads to a contradiction with \eqref{Appr Lem contrad eq} for $j$ sufficiently large, which completes the proof.

\end{proof}

\section{Optimal gradient estimates}\label{grad reg}
 
In this section, we deliver the optimal gradient regularity estimates for solutions to \eqref{main eq} under a smallness regime. We start off by fixing an aimed H\"older continuity exponent for the gradient of $u$. Let us fix 
\begin{equation}\label{optmal holder exponent}
    \beta \in ( 0, \alpha_0 ) \cap \left( 0, \frac{1}{1 + \gamma}\right], 
\end{equation}
where $\alpha_0$ and $\gamma>0$. We aim to show that $u \in C^{1,\beta}$ at the origin. It is standard to pass from pointwise estimates to interior regularity. The main result of this section is the following:

\begin{prop}\label{geometric estimate}
     Let  $u \in C(B_1)$ be a normalized viscosity solution to 
     $$
        |\nabla u |^{\gamma}F(x, D^2 u ) + h(x)|\nabla u |^m = f(x).
     $$ 
     Assume that $f \in L^{\infty}(B_1)$, $h \in C(B_1)\cap L^{\infty}(B_1)$,  and $F$ satisfy \eqref{unif ellipticity} and \eqref{coeff cont}. There exist universal numbers $\nu_{*}  > 0$ and $0 < \rho < 1/2$, depending additionally only upon $\beta$ such that if 
\begin{equation}\label{smallness discrete}
    \| {\rm osc}_{F}\|_{L^{\infty}(B_1)}  +   \|f\|_{L^{\infty}(B_1)} + \| h\|_{L^{\infty}(B_1)}  < \nu_{*}
\end{equation}
    then  $u$ is $C^{1,\beta}$ at the origin with estimate 
    \begin{equation}
        \sup\limits_{B_{r}} |u(x) - [u(0) + \nabla u(0) \cdot x]) | \le  \bar{C} r^{1+ \beta} , \quad \forall \,\,   0 < r \le \rho, 
    \end{equation}
      where $\bar{C} > 0$ is a universal constant. 
\end{prop}

\begin{proof}
The strategy of the proof of the proposition consists of finding a sequence of affine functions $\ell_j(x) = a_j + \vec{b_j} \cdot x$ satisfying
    \begin{equation}\label{osc with affine func}
    \sup\limits_{B_{\rho^j}} |u(x) - \ell_j(x) | \le \rho^{j(1 + \beta)} ,
\end{equation}

\begin{equation}\label{conver of affine func}
    |a_{j} - a_{j-1} | \le C_1 \rho^{(j-1)(1 + \beta)} \quad \text{and} \quad |\vec{b}_{j} - \vec{b}_{j-1}| \le C_1\rho^{(j-1)\beta},
\end{equation}
for all $ j \ge 1$ as in \cite{Caff, Caff-Cabre, Teix2014, Teix2015}. As usual, we proceed to prove this by induction combined with convergence argument. To easy the presentation, we divide the proof into three steps. 
\smallskip

\noindent{\bf Step 1 - The choice of the constants.}
We begin with the definition 
\[
\ell_0 (x):= z(0) + \nabla z(0) \cdot x,
\]
where $z$ is the approximate function given by the Approximation Lemma \ref{Approx lemma} for a $\delta>0$ to be chosen. From the regularity theory available for $F$-harmonic functions, i.e., solutions to $F(0, D^2 z) = 0$, there exist universal constants $\alpha_0 \in (0,1) $ and   $C_1  = C_1(n, \lambda, \Lambda) > 1 $ such that
$$
\sup_{B_{\rho}} |z(x) - \ell_0(x) | \le \ C_1 {\rho}^{1 + \alpha_0}, \quad \, \text{for all} \, \, 0 < \rho \le 1/2,
$$
and 
$$
    |a_0| + | \vec{b}_0| \le C_1,
$$
where $z(0) = a_0$ and $\vec{b}_0 = \nabla z(0)$, see e.g., \cite{Caff-Cabre}. In addition, we have 
\[
\sup_{B_{\rho}}  |u(x) - z(x) | < \delta.
\]
Applying the triangle inequality, we get
\[
\begin{array}{ccl}
     \sup_{B_{\rho}}  |u(x) - \ell_0(x) | & \le &   \sup_{B_{\rho}} |u(x) - z(x) | + \sup_{B_{\rho}}  |z(x) - \ell_0(x) | \vspace{0.2cm} \\
     & \le & \delta + C_1 {\rho}^{1 + \alpha_0}. \\
\end{array}
\]
Since $\beta$ fixed as in \eqref{optmal holder exponent} and $\rho< 1/2$ so small that
\[
C_1 \rho^{1+ \alpha_0} \le \frac{1}{2} \rho^{1+\beta}.
\]
We proceed by making the following universal choices.  
\begin{equation}\label{choice of rho}
   0 < \rho \le \left( \frac{1}{2C_1} \right)^{\frac{1}{\alpha_0 - \beta}} \quad \text{and} \quad \delta:=  \frac{1}{2} \rho^{1+\beta}.
\end{equation}
Therefore, 
\[
\sup_{B_{\rho}}  |u(x) - \ell_0(x) | \le \rho^{1 + \beta}
\]
as desired. The choice of $\delta$ determines $\eta$ through the Approximation Lemma \ref{Approx lemma}.  Next, we choose $\nu_* > 0$ small enough so that 
\begin{equation}\label{choice of nu}
    \nu_* \cdot \left( \left( 4 C_1\right)^{(m-\gamma)_+} + 1 \right)\le \min \left\{\frac{\eta}{3}, \frac{C_0}{2}\right\}
\end{equation}
where $C_0$ is the universal constant from Proposition \ref{compacness}. 
\smallskip

\noindent{\bf Step 2 - Iteration schemes.} Suppose by induction \eqref{osc with affine func} and \eqref{conver of affine func} has been verified for $j = 1, \dots, k$. Next,  define 
\begin{equation} \label{apprx aplied to v_k}
 v_k (x) = \frac{u(\rho^k x) - \ell_k(\rho^k x)}{\rho^{k(1 + \beta)}}, \quad \text{for} \, \, x \in B_1.    
\end{equation}
   
One can check that $v_k$ is a normalized solution to 
\begin{equation}\label{ind step eq}
    \left| \nabla v_k + \rho^{-k\beta} \vec{b}_k \right|^{\gamma} F_k(x, D^2 v_k ) + h_k (x)\left| \nabla v_k + \rho^{-k\beta} \vec{b}_k \right|^m = f_k (x)
\end{equation}
where 
$$
    F_k (x, D^2 v_k(x)): = \rho^{k(1 - \beta)} F(\rho^k x, \rho^{-k(1 - \beta)} D^2 v_k( x) )
$$
is a $(\lambda, \Lambda)$-elliptic operator and 
\[
{\rm osc}_{F_k}(x, y) =  \rho^{k (1-\beta)} {\rm osc}_{F}({\rho}^k x, {\rho}^k y),
\]
$$
    h_k(x) = \rho^{k[1 - \beta(1  + \gamma - m) ]} h(\rho^k x ) 
$$
and, 
$$
    f_k(x): = \rho^{k [ 1 - \beta(1+  \gamma)  ]} f(\rho^k x ).
$$
Next, by writing 
$$
     \vec{q}_k: = \rho^{-k\beta} \vec{b}_k , 
$$
In what follows, we want to estimate
$$
     \| h_k \|_{L^{\infty} (B_1)} \left( |\vec{q}_k|^{(m -\gamma)_+} + 1 \right). 
$$
If $0< m \le  \gamma$, then $(m - \gamma)_+ = 0$ and thus, 
\begin{eqnarray} \label{hamiltonian estimative_1}
    \| h_k \|_{L^{\infty} (B_1)} \left( |\vec{q}_k|^{(m -\gamma)_+} + 1 \right) &\le & \| h\|_{L^{\infty}(B_1)} \rho^{k[1 - \beta(1  + \gamma - m) ] - k\beta(m-\gamma)_+ } \left(|\vec{b}_k |^{(m -\gamma)_+} + 1 \right) \nonumber \\ 
    &=& 2 \| h\|_{L^{\infty}(B_1)} \rho^{k[1 - \beta(1  + \gamma - m) ] }  \nonumber \\ 
    &\le& 2 \| h\|_{L^{\infty}(B_1)}.
    \end{eqnarray}
Notice that $1 - \beta(1  + \gamma - m)>0$ by the choice of $\beta$ in \eqref{optmal holder exponent}. On the other hand, if $ \gamma < m \le 1 + \gamma$, then $(m - \gamma)_+ = m - \gamma$
\begin{eqnarray}\label{hamiltonian estimative}
    \| h_k \|_{L^{\infty} (B_1)} \left( |\vec{q}_k|^{(m -\gamma)} + 1 \right) &\le & \| h\|_{L^{\infty}(B_1)} \rho^{k[1 - \beta + \beta( m- \gamma) ] - k\beta(m-\gamma)} \left(|\vec{b}_k |^{(m -\gamma)} + 1 \right) \nonumber \\
    &=& \| h\|_{L^{\infty}(B_1)}  \rho^{k(1 -\beta)} \left(|\vec{b}_k |^{(m -\gamma)} + 1 \right) \nonumber \\
    &\le& \| h\|_{L^{\infty}(B_1)} \left(  |\vec{b}_k |^{(m -\gamma)} + 1 \right)
\end{eqnarray}
By the induction hypothesis, \eqref{conver of affine func} holds, for $1, \dots , k$, hence
\begin{eqnarray}
    |\vec{b}_k | &= &\left| \vec{b}_0 + \sum\limits_{l = 1}^{k} (\vec{b}_{l} - \vec{b}_{l-1})\right| \nonumber \\
    &\le&   |\vec{b}_0| + \sum\limits_{l = 1}^{k} |\vec{b}_{l} - \vec{b}_{l-1} | {\color{red}}\nonumber \\
    &\le & C_1  + \sum\limits_{l = 1}^{k} C_1 \rho^{(l-1)\beta}\nonumber  \\
    &\le & C_1  + \frac{C_1}{1 - \rho^\beta} \nonumber\\
    &\le & \frac{2 C_1}{1 - \rho^\beta} \nonumber\\
     &\le & 4 C_1, \nonumber\\
\end{eqnarray}
since $0 <\rho < 1/2$. Therefore, putting  \eqref{hamiltonian estimative_1} and \eqref{hamiltonian estimative} and combining with the fact that $ \| h\|_{L^{\infty}(B_1)} \le \nu_*$, where $\nu_*$ is the number given by the choice made in \eqref{choice of nu}, we have the following estimate
\begin{eqnarray} \label{k-step smallness regime}
   \| h_k \|_{L^{\infty} (B_1)} \left( |\vec{q}_k|^{(m -\gamma)_+} + 1 \right) 
   &\le & \| h\|_{L^{\infty}(B_1)} \left( \left( 4C_1 \right)^{(m -\gamma)_+} + 1 \right) \nonumber\\
   & \le &  \nu_*  \left( \left( 4C_1 \right)^{(m -\gamma)_+} + 1 \right) \nonumber \\
   &\le & \min \left\{\frac{\eta}{3}, \frac{C_0}{2}\right\}.
\end{eqnarray}
Moreover, 
$$
    \| \mbox{osc}_{F_{k}} \|_{L^{\infty}(B_1)} = \rho^{k (1-\beta)} \|{\rm osc}_{F} \|_{L^{\infty}\left(B_{\rho^k}\right)} \le \| \mbox{osc}_F \|_{L^{\infty}(B_1)} \le \nu_* \le  \frac{\eta}{3}.
$$
Also, since $\rho \ll 1$ and $k [ 1 - \beta(1+  \gamma)] >0$, we have
\[
 \|f_k \|_{L^{\infty}(B_1)} = \rho^{k [ 1 - \beta(1+  \gamma)  ]} \|f \|_{L^{\infty}\left(B_{\rho^k}\right)} \le \| f\|_{L^{\infty}(B_1)} \le \nu_* \le \frac{\eta}{3}. 
\]
Hence, 
\begin{equation*}\label{smallness regime}
 \| {\rm osc}_{F_k} \|_{L^{\infty}(B_1)}  +   \|f_k\|_{L^{\infty}(B_1)} + \| h_k\|_{\infty} ( |\vec{q_k}|^{(m - \gamma)_+} + 1) < \eta,
\end{equation*}
holds. Therefore,  $v_k$ is under the assumptions of the Approximation Lemma~\ref{Approx lemma}. Hence, we can find a function $\tilde{z}$, solution to 
$$
    \mathcal{F} (0, D^2 \tilde{z} ) = 0 \quad \text{in} \quad B_{3/4}
$$
such that
$$
    \sup_{B_{1/2}}| v_k(x) - \tilde{z}(x) | \le \delta .
$$

As in Step 1, there exists $\tilde{\ell}(x)$ such that
\begin{equation} \label{ineq:linear-approximation}
    \sup_{B_{\rho}}  |v_k(x) - \tilde{\ell}(x) | \le \rho^{1 + \beta}, 
\end{equation}
where $\tilde{\ell}(x) =\tilde{z}(0) + \nabla \tilde{z}(0) \cdot x $. By using the definition of $v_k $ in \eqref{apprx aplied to v_k} and scaling  \eqref{ineq:linear-approximation} back to $u$ we get,
$$
    \sup\limits_{B_{\rho^{k+1}}} | u(x) - \ell_{k+1} (x) | \le \rho^{(k+1)(1 + \beta)},
$$
where $\ell_{k+1} (x)  = a_{k+1} + \vec{b}_{k+1} \cdot x$
with 
$$
    a_{k+1} = a_k + \rho^{(k+1) \beta} \tilde{z}(0)\quad \text{and} \quad \vec{b}_{k+1} = \vec{b}_k + \rho^{k\beta}  \nabla \tilde{z}(0). 
$$
Moreover, 
$$
    |a_{k+1} - a_k | = | \rho^{(k+1) \beta} \tilde{z}(0)| \le C_1 \rho^{(k+1)\beta},
$$
and
$$
    |\vec{b_{k+1}} - \vec{b}_k | = | \rho^{k\beta}  \nabla \tilde{z}(0) | \le C_1 \rho^{k\beta} .
$$
Hence, we have proven that \eqref{osc with affine func} and \eqref{conver of affine func} hold for all $j \in \mathbb{N}$.

\smallskip

\noindent{\bf Step 3 - Convergence analysis.} It follows from Step 2 that  \eqref{osc with affine func} and \eqref{conver of affine func} hold for all $j \in \mathbb{N}$. In particular, $(a_j)_{j \in \mathbb{N}}$ and $(\vec{b}_j)_{j \in \mathbb{N}}$ are Cauchy sequences, and therefore converge, that is,  
$$
    a_j \to a_{*} \in \mathbb{R}  \quad \text{and} \quad \vec{b}_j \to \vec{b}_{*} \in \mathbb{R}^n
$$
with
$$
    |a_{*} - a_j| \le \frac{C_1 \rho^{j(1+ \beta)}}{1 - \rho^{\beta}} \quad \text{and} \quad  |\vec{b}_{*} - \vec{b}_j| \le \frac{C_1 \rho^{ j\beta}}{1 - \rho^{\beta}} .
$$
Therefore, given any $ 0 < r \le \rho$, take $ j \in \mathbb{N}$ such that $\rho^{j + 1} < r \le \rho^j$. Then for $\ell_{*} (x) = a_{*} + \vec{b}_{*} \cdot x$ we have the following: 
\begin{eqnarray}
     \sup\limits_{B_r} | u(x) - \ell_{*} (x)  | &\le&  \sup\limits_{B_{\rho^j} }| u(x) - \ell_{*} (x)  | \nonumber \\
     &\le& \sup\limits_{B_{\rho^{j}}} | u(x) - \ell_{j}(x) |  + \sup\limits_{B_{\rho^{j}}} | \ell_j (x) - \ell_{*} (x)  | \nonumber \\ 
     &\le& \sup\limits_{B_{\rho^{j}}} | u(x) - \ell_{j}(x) |  + \sup\limits_{B_{\rho^{j}}} |a_{*} - a_j| + \sup\limits_{B_{\rho^{j}}} |\vec{b}_{*} - \vec{b}_j||x|\nonumber \\ 
     &\le& \rho^{j(1 + \beta)} + 2 \frac{C_1 \rho^{ j(1 + \beta)}}{1 - \rho^{1 + \beta}} \nonumber \\
     &\le& \frac{1}{\rho^{1+ \beta}}\left( 1 + \frac{2C_1 }{1 - \rho^{1 + \beta}} \right) r^{1+ \beta} \nonumber \\ 
      &\le& \bar{C} r^{1+ \beta} \nonumber \\ 
\end{eqnarray}
and the proof is complete. 
\end{proof}

We now conclude the proof of Theorem \ref{main thm} and present the proof of Corollary \ref{main cor}. 
   
 \begin{proof}[\it Proof of Theorem \ref{main thm}]
As discussed in \cite[Section 7]{ART},  the scaling features of the equation allow us to reduce the proof of Theorem \ref{main thm} to the hypotheses of Proposition \ref{geometric estimate}.  Indeed, if  $v \in C(B_1)$ is a viscosity solution to 
$$
    \Phi(x, \nabla v) F(x, D^2 v) + h(x) |\nabla v|^m = f(x)\quad \text{in} \quad B_1,
$$
where $\Phi$ satisfy \eqref{degeneracy law}, $F$ is a $(\lambda, \Lambda)$-elliptic operator with continuous coefficients satisfying \eqref{unif ellipticity} and \eqref{coeff cont},  then by defining 
$$
    u(x) =\kappa v(rx) 
$$
for parameters $r, \tau > 0$ to be determined, we readily check that $u$ solves
    $$
   \Phi_{r,\kappa} (x, \nabla u) F_{r, \kappa} (x, D^2 u)  + h_{r,\kappa} (x) |\nabla u |^m =  f_{r, \kappa} (x)
 $$
 where 
 $$
    F_{r, \kappa} (x, M) = \kappa r^2 F(rx , \kappa^{-1} r^{-2} M ),
$$ 
is a uniformly elliptic operator with the same ellipticity constants as $F$. 
$$
   \Phi_{r, \kappa} (x, \vec{b} ) = (\kappa r)\Phi (rx, (\kappa r)^{-1}\vec{b} )
$$ satisfies the degeneracy condition \eqref{degeneracy law}, with the same constants. 
$$
    h_{r,\kappa} (x) = r^{2 + \gamma - m} \kappa^{1 + \gamma - m} h(rx),
$$
and 
$$
    f_{r,\kappa} (x) = r^{2 + \gamma} \kappa^{1 + \gamma} f(rx).
$$ 
Next, for $m < 1 + \gamma$, by choosing 
$$
    \kappa =  \min\left\{ 1, \left(  \| v\|_{L^{\infty}(B_1) }  + \| h\|_{L^{\infty}(B_1)}^{\frac{1}{1+ \gamma - m}} + \|f \|_{L^{\infty} (B_1)}^{\frac{1}{1+\gamma}} \right)^{-1} \right\},
$$  
we see that $\|u\|_{L^{\infty}(B_1)} \le 1 $. And by choosing  
$$
    r:= \min\left\{  \frac{1}{2}, \left( \frac{\nu_*}{3}  \right)^{\frac{1}{2 + \gamma}}, \left( \frac{\nu_*}{3}\right)^{\frac{1}{2+\gamma - m}} , \left( \frac{\nu_*}{3C_F} \right)^{1/\tau} \right\}, 
$$
one easily verifies that $u$ is under the smallness assumptions of Proposition \ref{geometric estimate}. For $m = 1+\gamma$, our choice of $\kappa$ and $r$ is as follows: 

$$
    \kappa = \min\left\{ 1,  \left(  \| v\|_{L^{\infty}(B_1) }\right)^{-1} \right \} 
$$
and 
$$
    r:= \min\left\{  \frac{1}{2}, \left( \frac{\nu_*}{3 \|f \|_{L^{\infty} (B_1)}^{\frac{1}{1+\gamma}} }  \right)^{\frac{1}{2 + \gamma}}, \left( \frac{\nu_*}{3 \| h \|_{L^{\infty} (B_1)}} \right)^{\frac{1}{2+\gamma - m}} ,  \left( \frac{\nu_*}{3C_F} \right)^{1/\tau} \right\}.
$$
Once again, one can verify that these choices guarantee that $u$ is under the hypothesis of Proposition \ref{geometric estimate}. 
The corresponding estimate for $v$ now follows directly. 
\end{proof}

We finish this section by proving the quantitative result stated Corollary \ref{main cor}.
\begin{proof}[\it Proof of Corollary \ref{main cor}]
Since $F$ is concave,  solutions to $F(0, D^2 z) =0$ are locally of class $C^{2, \alpha}$ for some $\alpha \in (0, 1)$ by the classical Evans-Krylov theory. We now follow the approach of Proposition~\ref{geometric estimate}, with the only difference that in Step 1 we assume \(\beta < 1\) instead of \(\beta < \alpha_0\).

\end{proof}

\noindent {\bf Acknowledgments:} The authors sincerely thank the Department of Mathematics at Iowa State University for its support and hospitality, where this project was initiated. P\^edra D. S. Andrade is partially supported by the Portuguese government through FCT - Funda\c c\~ao para a Ci\^encia e a Tecnologia, I.P., under the project UIDP/00208/2020 (DOI: 10.54499/UIDP/00208/2020) and part by the Austrian Science Fund (FWF) [10.55776/P36295]. For open access purposes, the author has applied a CC BY public copyright license to any author accepted manuscript version arising from this submission. 

\bibliographystyle{plain}
\bibliography{AndrNasc_Arxiv1}

\begin{thebibliography}{10}

\bibitem{APPT}
P\^edra Andrade, Daniel Pellegrino, Edgard Pimentel, and Eduardo Teixeira.
\newblock {$C^1$}-regularity for degenerate diffusion equations.
\newblock {\em Adv. Math.}, 409:Paper No. 108667, 34, 2022.

\bibitem{ART}
Dami\~ao~J. Ara\'ujo, Gleydson Ricarte, and Eduardo Teixeira.
\newblock Geometric gradient estimates for solutions to degenerate elliptic
  equations.
\newblock {\em Calc. Var. Partial Differential Equations}, 53(3-4):605--625,
  2015.

\bibitem{AttBarles15}
Amal Attouchi and Guy Barles.
\newblock Global continuation beyond singularities on the boundary for a
  degenerate diffusive {H}amilton-{J}acobi equation.
\newblock {\em J. Math. Pures Appl. (9)}, 104(2):383--402, 2015.

\bibitem{Baasandorj-Byun-Oh2023}
Sumiya Baasandorj, Sun-Sig Byun, and Jehan Oh.
\newblock {$C^1$} regularity for some degenerate/singular fully nonlinear
  elliptic equations.
\newblock {\em Appl. Math. Lett.}, 146:Paper No. 108830, 10, 2023.

\bibitem{Barles10}
Guy Barles.
\newblock A short proof of the {$C^{0,\alpha}$}-regularity of viscosity
  subsolutions for superquadratic viscous {H}amilton-{J}acobi equations and
  applications.
\newblock {\em Nonlinear Anal.}, 73(1):31--47, 2010.

\bibitem{BD04}
Isabeau Birindelli and Fran\c~coise Demengel.
\newblock Comparison principle and {L}iouville type results for singular fully
  nonlinear operators.
\newblock {\em Ann. Fac. Sci. Toulouse Math. (6)}, 13(2):261--287, 2004.

\bibitem{BD06}
Isabeau Birindelli and Fran\c{c}oise Demengel.
\newblock First eigenvalue and maximum principle for fully nonlinear singular
  operators.
\newblock {\em Adv. Differential Equations}, 11(1):91--119, 2006.

\bibitem{BD10}
Isabeau Birindelli and Fran\c{c}oise Demengel.
\newblock Regularity and uniqueness of the first eigenfunction for singular
  fully nonlinear operators.
\newblock {\em J. Differential Equations}, 249(5):1089--1110, 2010.

\bibitem{BD2014}
Isabeau Birindelli and Fran\c{c}oise Demengel.
\newblock $\mathscr{C}^{1,\beta}$ regularity for {D}irichlet problems
  associated to fully nonlinear degenerate elliptic equations.
\newblock {\em ESAIM Control Optim. Calc. Var.}, 20(4):1009--1024, 2014.

\bibitem{BD2016}
Isabeau Birindelli and Fran\c{c}oise Demengel.
\newblock Fully nonlinear operators with {H}amiltonian: {H}\"older regularity
  of the gradient.
\newblock {\em NoDEA Nonlinear Differential Equations Appl.}, 23(4):Art. 41,
  17, 2016.

\bibitem{BDL2019}
Isabeau Birindelli, Fran\c{c}oise Demengel, and Fabiana Leoni.
\newblock {$\mathcal{C}^{1,\gamma}$} regularity for singular or degenerate
  fully nonlinear equations and applications.
\newblock {\em NoDEA Nonlinear Differential Equations Appl.}, 26(5):Paper No.
  40, 13, 2019.

\bibitem{Caff}
Luis Caffarelli.
\newblock Interior a priori estimates for solutions of fully nonlinear
  equations.
\newblock {\em Ann. of Math. (2)}, 130(1):189--213, 1989.

\bibitem{Caff-Cabre}
Luis Caffarelli and Xavier Cabr\'e.
\newblock {\em Fully nonlinear elliptic equations}, volume~43 of {\em American
  Mathematical Society Colloquium Publications}.
\newblock American Mathematical Society, Providence, RI, 1995.

\bibitem{CIL92}
Michael Crandall, Hitoshi Ishii, and Pierre-Louis Lions.
\newblock User's guide to viscosity solutions of second order partial
  differential equations.
\newblock {\em Bull. Amer. Math. Soc. (N.S.)}, 27(1):1--67, 1992.

\bibitem{SN}
Jo\~ao~Vitor da~Silva and Gabrielle Nornberg.
\newblock Regularity estimates for fully nonlinear elliptic {PDE}s with general
  {H}amiltonian terms and unbounded ingredients.
\newblock {\em Calc. Var. Partial Differential Equations}, 60(6):Paper No. 202,
  40, 2021.

\bibitem{Delarue2010}
Fran\c{c}ois Delarue.
\newblock Krylov and {S}afonov estimates for degenerate quasilinear elliptic
  {PDE}s.
\newblock {\em J. Differential Equations}, 248(4):924--951, 2010.

\bibitem{ImbertSilvestre2013}
Cyril Imbert and Luis Silvestre.
\newblock {$C^{1,\alpha}$} regularity of solutions of some degenerate fully
  non-linear elliptic equations.
\newblock {\em Adv. Math.}, 233:196--206, 2013.

\bibitem{Imbert-Silvestre2016}
Cyril Imbert and Luis Silvestre.
\newblock Estimates on elliptic equations that hold only where the gradient is
  large.
\newblock {\em J. Eur. Math. Soc. (JEMS)}, 18(6):1321--1338, 2016.

\bibitem{IL}
Hitoshi Ishii and Pierre-Louis Lions.
\newblock Viscosity solutions of fully nonlinear second-order elliptic partial
  differential equations.
\newblock {\em J. Differential Equations}, 83(1):26--78, 1990.

\bibitem{JPU22}
David Jesus, Edgard Pimentel, and Jos\'e{}~Miguel Urbano.
\newblock Fully nonlinear {H}amilton-{J}acobi equations of degenerate type.
\newblock {\em Nonlinear Anal.}, 227:Paper No. 113181, 15, 2023.

\bibitem{Krylov-Safonov1979}
Nicolai Krylov and Mikhail Safonov.
\newblock An estimate for the probability of a diffusion process hitting a set
  of positive measure.
\newblock {\em Dokl. Akad. Nauk SSSR}, 245(1):18--20, 1979.

\bibitem{Krylov-Safonov1980}
Nicolai Krylov and Mikhail Safonov.
\newblock A property of the solutions of parabolic equations with measurable
  coefficients.
\newblock {\em Izv. Akad. Nauk SSSR Ser. Mat.}, 44(1):161--175, 239, 1980.

\bibitem{LP16}
Tommaso Leonori and Alessio Porretta.
\newblock Large solutions and gradient bounds for quasilinear elliptic
  equations.
\newblock {\em Comm. Partial Differential Equations}, 41(6):952--998, 2016.

\bibitem{NV2007}
Nikolai Nadirashvili and Serge Vl\u{a}du\c{t}.
\newblock Nonclassical solutions of fully nonlinear elliptic equations.
\newblock {\em Geom. Funct. Anal.}, 17(4):1283--1296, 2007.

\bibitem{NV2008}
Nikolai Nadirashvili and Serge Vl\u{a}du\c{t}.
\newblock Singular viscosity solutions to fully nonlinear elliptic equations.
\newblock {\em J. Math. Pures Appl. (9)}, 89(2):107--113, 2008.

\bibitem{NV2011}
Nikolai Nadirashvili and Serge Vl\u{a}du\c{t}.
\newblock Singular solutions of {H}essian fully nonlinear elliptic equations.
\newblock {\em Adv. Math.}, 228(3):1718--1741, 2011.

\bibitem{Pucci}
Carlo Pucci.
\newblock Operatori ellittici estremanti.
\newblock {\em Ann. Mat. Pura Appl. (4)}, 72:141--170, 1966.

\bibitem{Safonov}
Mikhail Safonov.
\newblock Harnack's inequality for elliptic equations and {H}\"older property
  of their solutions.
\newblock {\em Zap. Nauchn. Sem. Leningrad. Otdel. Mat. Inst. Steklov. (LOMI)},
  96:272--287, 312, 1980.

\bibitem{Teix2015}
Eduardo Teixeira.
\newblock Hessian continuity at degenerate points in nonvariational elliptic
  problems.
\newblock {\em Int. Math. Res. Not. IMRN}, 2015(16):6893--6906, 2014.

\bibitem{Teix2014}
Eduardo Teixeira.
\newblock Universal moduli of continuity for solutions to fully nonlinear
  elliptic equations.
\newblock {\em Arch. Ration. Mech. Anal.}, 211(3):911--927, 2014.

\end{thebibliography}

\end{document}